\documentclass[a4paper,11pt]{article}
\usepackage{latexsym,amssymb,enumerate,amsmath,epsfig,amsthm}
\usepackage[margin=1in]{geometry}
\usepackage{setspace,color}
\usepackage{tikz,url}
\usepackage{floatrow}
\usepackage{subfig}
\usepackage[ruled]{algorithm2e}

\newcommand{\x}{\mathbf{x}}
\newcommand{\y}{\mathbf{y}}

\newcommand{\n}{\mathbf{n}}
\newcommand{\bu}{\mathbf{u}}
\newcommand{\reminder}[1]{#1}

\newcommand{\reminderre}[1]{#1}

\newtheorem{thm}{Theorem}[section]

\newtheorem{cor}{Corollary}[section]
\newtheorem{defn}{Definition}[section]

\newtheorem{remark}{Remark}[section]

\newcommand {\pic}{\setlength{\unitlength}{1cm}
\begin{picture}(4,0)
\thicklines \put(-0.5,0){\line(1,0){12}}
\end{picture}
}

\begin{document}

\title{Eulerian Methods for Visualizing Continuous Dynamical Systems using Lyapunov Exponents}
\author{
Guoqiao You\thanks{School of Science, Nanjing Audit University, Nanjing, 211815, China. Email: {\bf 270217@nau.edu.cn}}
\and
Tony Wong\thanks{Department of Mathematics and Institute of Applied Mathematics, University of British Columbia, Vancouver V6T 1Z2, BC, Canada. Email: {\bf tonyw@math.ubc.ca}}
\and
Shingyu Leung\thanks{Department of Mathematics, the Hong Kong University of Science and Technology, Clear Water Bay, Hong Kong. Email: {\bf masyleung@ust.hk}}
}
\maketitle

\begin{abstract}
\reminder{We propose a new Eulerian numerical approach for constructing the forward flow maps in continuous dynamical systems. The new algorithm improves the original formulation developed in \cite{leu11,leu13} so that the associated partial differential equations (PDEs) are solved forward in time and, therefore, the \textit{forward} flow map can now be determined \textit{on the fly}. Due to the simplicity in the implementations, we are now able to efficiently compute the unstable coherent structures in the flow based on quantities like the finite time Lyapunov exponent (FTLE), the finite size Lyapunov exponent (FSLE) and also a related infinitesimal size Lyapunov exponent (ISLE). When applied to the ISLE computations, the Eulerian method is in particularly computational efficient.} For each separation factor $r$ in the definition of the ISLE, typical Lagrangian methods require to shoot and monitor an individual set of ray trajectories. If the scale factor in the definition is changed, these methods have to restart the whole computations all over again. The proposed Eulerian method, however, requires to extract only an isosurface of a volumetric data for an individual value of $r$ which can be easily done using any well-developed efficient interpolation method or simply an isosurface extraction algorithm. Moreover, we provide a theoretical link between the FTLE and the ISLE fields which explains the similarity in these solutions observed in various applications.
\end{abstract}

\section{Introduction}

It is an important task to visualize, understand and then extract useful information in complex continuous dynamical systems from many science and engineering fields including flight simulation \reminderre{\cite{carmoh08,tanchahal10}}, wave propagation \reminderre{\cite{tanpea10}}, bio-inspired fluid flow motions \reminderre{\cite{lipmoh09,grerowsmi10,lukyanfau10}}, ocean current model \reminderre{\cite{lekleo04,shalekmar05}}, and etc. It is, therefore, natural to see many existing tools to study dynamical systems. For example, a lot of work in dynamical systems focuses on understanding different types of behaviors in a model, such as elliptical zones, hyperbolic trajectories, chaotic attractors, mixing regions, to name just a few.

One interesting tool is to partition the space-time domain into subregions based on certain quantity measured along with the passive tracer advected according to the associated dynamical system. Because of such a Lagrangian property in the definition, the corresponding partition is named the Lagrangian coherent structure (LCS). One commonly used quantity is the so-called finite time Lyapunov exponent (FTLE) \cite{hal01,hal01b,hal11,karhal13} which measures the rate of separation of a passive tracer with an infinitesimal perturbation in the initial condition, over a finite period of time. Another popular diagnostic of trajectory separation in dynamical systems is the finite-size Lyapunov exponent (FSLE) \cite{abcpv97,hlht11,cenvul13} which is especially popular in the applications from the oceanography. Instead of introducing an infinitesimal perturbation in the initial condition of the fluid particles, FSLE considers finite perturbation. It does not measure the rate of separation but computes the time required to separate two adjacent tracers up to certain distance.

Since all these quantities are long treated as a Lagrangian property of a continuous dynamical system, most, if not all, numerical methods are developed based on the traditional Lagrangian ray tracing method by solving the ODE system using any well-developed numerical integrator. These approaches, however, require the velocity field defined at arbitrary locations in the whole space depending on the location of each individual particle. This implies that one has to in general implement some interpolation routines in the numerical code. Unfortunately, it could be numerically challenging to develop an interpolation approach which is computationally cheap, \reminder{monotone, and high order accurate}.

An Eulerian approach to the flow map computation for the FTLE has been first proposed in \cite{leu11} which has incorporated the level set method \cite{oshset88}. Based on the phase flow method \cite{canyin06,leuqia09}, we have developed a backward phase flow method for the Eulerian FTLE computations in \cite{leu13}. Based on these partial differential equation (PDE) based algorithms, more recently, we have developed in \cite{youleu14} an efficient Eulerian numerical approach for extracting invariant sets in a continuous dynamical system in the extended phase space (the $\x-t$ space). In \cite{youleu14b}, we have proposed a simple algorithm called VIALS which determines the average growth in the surface area of a family of level surfaces. This Eulerian tool relates closely to the FTLE and provides an alternative for understanding complicated dynamical systems.

In this paper, we are going to further develop these Eulerian tools for continuous dynamical systems. There are several main contributions of this paper. We will first improve the numerical algorithm for \textit{forward} flow map construction. The papers \cite{leu11,leu13} have proposed to construct the \textit{forward} flow map defined on a fixed Cartesian mesh by solving the level set equations or the Liouville equations in the \textit{backward} direction. In particular, to compute the forward flow map from the initial time $t=0$ to the final time $t=T$, one needs to solve the Liouville equations \textit{backward} in time from $t=T$ to $t=0$ with the initial condition given at $t=T$. This implementation is numerically inconvenient, especially when incorporating with some computational fluid dynamic (CFD) solvers, since the velocity field is loaded from the current time $t=T$ backward in time to the initial time. This implies that the whole field at all time steps has to be stored in the desk which might not be practical at all. In this paper, we first propose an Eulerian approach to compute the \textit{forward} flow map \textit{on the fly} so that the PDE is solved \textit{forward} in time. 

\reminder{Developed using this improved algorithm, we propose a simple Eulerian approach to compute the ISLE.} For each of the separation factor $r$ in the definition of the ISLE, typical Lagrangian method requires to shoot an individual set of rays and to monitor the variations in these trajectories. If the scale $r$ is changed, these methods will require to restart the whole computations all over again. In this work, we are going to develop an efficient algorithm which provides the ISLE field for any arbitrary separation factor $r$. For an individual value of $r$, the method requires to extract only an isosurface of a volumetric data. This can be done easily by any well-developed interpolation algorithm or simply the function {\sf isosurface} in {\sf MATLAB}. Another main contribution of this work is to provide a theoretical link between the FTLE and the ISLE. Even though these two quantities are measuring different properties of the flow, it has been reported widely, such as \cite{ppss14}, that they give visually very similar solutions in many examples. We are going to show in particular that their ridges can be identified through the ridge of the largest eigenvalue of the associated Cauchy-Green deformation tensor of the flow. \reminder{To the best of our knowledge, this theoretical result is the first one quantitatively revealing the relationship between the FTLE ridges and the ISLE ridges.}

Indeed, comparisons between FTLE and FSLE have been made in various studies. The paper \cite{blrv01} has pointed out that the FTLE might be unable to recognize the boundaries between the chaotic and the large-scale mixing regime. The article \cite{karhal13} has argued that the FSLE has several limitations in Lagrangian coherence detection including aspects from local ill-posedness, spurious ridges and intrinsic jump-discontinuities. Nevertheless, the work \cite{ppss14} has stated that these two concepts might yield similar results, if properly calibrated, and could be even interchangeable in \reminder{many flow visualizatons such as oceanography and chemical kinetics.} In this paper, however, we would like to emphasize on the computational aspect of using these Lyapunov exponents, rather than on comparing the advantages and the disadvantages of these different methods. The preference for one or the other could be based on the tradition in various fields, or the availability of the computational resources.

This paper is organized as follows. In Section \ref{Sec:Background} we will give a summary of several important concepts including the FTLE, FSLE and ISLE, and also our original Eulerian formulations for computing the flow maps and the FTLE. Then, our proposed Eulerian algorithms will be given in Section \ref{Sec:Proposal}. We will then point out some theoretical properties and relationships between the FTLE and the ISLE in Section \ref{Sec:Relation}. Finally, some numerical examples will be given in Section \ref{Sec:Examples}.


\section{Background}
\label{Sec:Background}

In this section, we will summarize several useful concepts and methods, which will be useful for the developments that we are proposing. We first introduce the definition of various closely related Lyapunov exponents including the finite time Lyapunov exponent (FTLE) \cite{halyua00,hal01,hal01b,shalekmar05,lekshamar07}, the finite size Lyapunov exponent (FSLE) \cite{abccv97,abcpv97,letkan00} and also the infinitesimal size Lyapunov exponent (ISLE) \cite{karhal13}. Then we discuss typical Lagrangian approaches and summarize the original Eulerian approaches as in \cite{leu11,leu13,youleu14,youleu14b}. The discussions here, however, will definitely not be a complete survey. We refer any interested readers to the above references and thereafter.

\subsection{FTLE, FSLE and ISLE}

In this paper, we consider a continuous dynamical system governed by the following ordinary differential equation (ODE)
\begin{equation}
\x'(t;\x_0,t_0)=\bu(\x(t;\x_0,t_0),t) \label{Eqn:ODE}
\end{equation}
with the initial condition $\x(t_0;\x_0,t_0)=\x_0$. The velocity field $\bu:\Omega \times \mathbb{R} \rightarrow \mathbb{R}^d$ is a time dependent Lipschitz function where $\Omega \subset \mathbb{R}^d$ is a bounded domain in the $d$-dimensional space. To simplify the notation in the later sections, we collect the solutions to this ODE for all initial conditions in $\Omega$ at all time $t\in\mathbb{R}$ and introduce the flow map
$$
\Phi_a^b:\Omega \times \mathbb{R} \times \mathbb{R} \rightarrow \mathbb{R}^d
$$
such that $\Phi_a^b(\x_0)=\x(b;\x_0,a)$ represents the arrival location $\x(b;\x_0,a)$ of the particle trajectory satisfying the ODE (\ref{Eqn:ODE}) with the initial condition $\x(a;\x_0,a)=\x_0$ at the initial time $t=a$. This implies that the mapping will take a point from $\x(a;\x_0,a)$ at $t=a$ to another point $\x(b;\x_0,a)$ at $t=b$.

The finite time Lyapunov exponent (FTLE) \cite{halyua00,hal01,hal01b,shalekmar05,lekshamar07} measures the rate of separation between adjacent particles over a finite time interval with an infinitesimal perturbation in the initial location. Mathematically, consider the initial time to be 0 and the final time to be $t$, we have the change in the initial infinitesimal perturbation given by
\begin{eqnarray*}
\delta \mathbf{x}(t) &=& \Phi_0^t(\mathbf{x}+\delta \mathbf{x}(0)) - \Phi_0^t(\mathbf{x}) \nonumber\\
&=& \nabla \Phi_0^t(\mathbf{x}) \delta \mathbf{x}(0) + \mbox{ higher order terms} \, .
\end{eqnarray*}
The leading order term of the magnitude of this perturbation is given by
$$
\| \delta \mathbf{x}(t) \| = \sqrt{ \left<
\delta \mathbf{x}(0), [\nabla \Phi_0^t(\mathbf{x})]^* \nabla \Phi_0^t(\mathbf{x}) \delta \mathbf{x}(0)
\right> } \, .
$$
With the \reminder{Cauchy-Green strain tensor} $\reminder{C_0^t(\mathbf{x})}=[\nabla \Phi_0^t(\mathbf{x})]^* \nabla \Phi_0^t(\mathbf{x})$, we obtain the \reminder{maximum deformation}
$$
\max_{\delta \x(0)} \|\delta \x(t) \| = \sqrt{\lambda_{\max}[\reminder{C_0^t(\mathbf{x})]}} \| \mathbf{e}(0)\| =
e^{ \sigma_0^t(\x) |t|} \|\mathbf{e}(0) \| \, ,
$$
where $\mathbf{e}(0)$ is the eigenvector associated with the largest eigenvalue of the deformation tensor. Using this quantity, the finite-time Lyapunov exponent (FTLE) $\sigma_0^t(\mathbf{x})$ is defined as
\begin{equation}
\sigma_0^t(\mathbf{x}) = \frac{1}{|t|} \ln \sqrt{\lambda_{\max}[\reminder{C_0^t(\mathbf{x})}]} = \frac{1}{|t|} \ln \sqrt{\lambda_0^t(\x)} \, .
\label{Eqn:FTLE}
\end{equation}
where $\lambda_0^t(\x)=\lambda_{\max}(\reminder{C_0^{t}(\x)})$ denotes the largest eigenvalue of the Cauchy-Green tensor. The absolute value of $t$ in the expression reflects the fact that we can trace the particles either \textit{forward} or \textit{backward} in time. In the case when $t<0$, we are measuring the maximum stretch \textit{backward} in time and this corresponds to the maximum compression \textit{forward} in time. To distinguish different measures, we call $\sigma_0^t(\mathbf{x})$ the \textit{forward} FTLE if $t>0$ and the \textit{backward} FTLE if $t<0$.

A related concept is a widely used concept in oceanography called the finite size Lyapunov exponent (FSLE) which tries to measure the time it takes to separate two adjacent particles up to certain distance \cite{abcpv97,hlht11,cenvul13}. There are two length scales in the original definition of the FSLE. One is the initial distance between two particles which is denoted by $\epsilon$. The other one is the so-called separation factor, denoted by $r>1$. For any trajectory with \reminder{initial location} $\x(0;\x_1,0)=\x_1$ with $\|\x_0-\x_1\|=\epsilon$, we first determine the shortest time $\tau_r(\x_1)>0$ so that
$$
\|\x(\tau_r(\x_1);\x_0,0)-\x(\tau_r(\x_1);\x_1,0)\| = \| \Phi_0^{\tau_r(\x_1)}(\x_0)-\Phi_0^{\tau_r(\x_1)}(\x_1) \| = r \, .
$$
The FSLE at the point $\x_0$ is then defined as
$$
\gamma_{\epsilon,r}(\x_0,0)=\max_{\|\x_0-\x_1\|=\epsilon} \frac{ \ln r}{\lvert \tau_r(\x_1)\rvert} \, .
$$

Numerically, on the other hand, the maximization over the constraint $\|\x_0-\x_1\|=\epsilon$ might require special considerations and the corresponding optimization problem might not be easily solved by typical optimization algorithms. One simple approximation is to consider only $2^d$-neighbors of $\x_0$ with each point $\x_1$ obtained by perturbing one coordinate of $\x_0$ with distance $\epsilon$.

To avoid introducing two \reminder{separate} length scales ($\epsilon$ and $r$) and the approximation by $2^d$ neighbors in the final optimization step, one theoretically takes the limit as the length scale $\epsilon$ tends to 0. The FSLE can then be reduced to the infinitesimal size Lyapunov exponent (ISLE) as defined in \cite{karhal13} given by
\begin{equation}
\gamma_r(\mathbf{x},0)=\frac{\ln r}{\lvert \tau_r(\x) \rvert}
\label{Eqn:ISLE}
\end{equation}
where $\lvert \tau_r(\x) \rvert$ is the shortest time for which \reminder{$\sqrt{\lambda_0^{\tau_r(\mathbf{x})}(\x)}=r$}.

Even with tremendous theoretical development, there are not many discussions on the numerical implementations. Because the quantity is defined using Lagrangian particle trajectories, most numerical algorithms for computing the FTLE or the FSLE are developed based on the ray tracing method which tries to solve the flow system (\ref{Eqn:ODE}) using a numerical integrator. In the two dimensional case, for example, the flow map is first determined by solving the ODE system with the initial condition imposed on a uniform mesh $\x=\x_{i,j}$. Then one can apply simple finite difference to determine the deformation tensor and, therefore, its eigenvalues. The computation of the FTLE is a little easier. For a given final time, one simply solves the ODE system up to that particular time level. Then the FTLE at each grid location can be computed directly from these eigenvalues. The ISLE, on the other hand, is a little more involved since for a particular separation factor, $r$, one has to continuously monitor the change in these eigenvalues and record the precise moment when it reaches the thresholding value.

Computational complexity could be one major concern of these numerical tools. Let $M=O(\Delta x^{-1})$ and $N=O(\Delta t^{-1})$ be the number of grid points in each physical direction and in the temporal direction, respectively. The total number of operations required to compute the flow map $\Phi_{t_0}^{t_N}(\x_{i,j})$ is, therefore, $O(M^2N)$. Note however that such complexity is needed for determining the FTLE or the ISLE on one single time level. For all time levels, the computational complexity for these usual Lagrangian implementation is $O(M^2N^2)$. There are several nice ideas to improve the computational time. One natural idea is to concentrate the computations near the locations where the solutions have large derivatives by implementing an adaptive mesh strategy \cite{sadpei07,ggth07,lekros10}. Another interesting approach is to avoid the computation of the overall map $\Phi_{t_0}^{t_N}$ but to numerically decompose it into sub-maps $\Phi_{t_n}^{t_{n+1}}$ \cite{brurow10}. Then, the extra effort from $\Phi_{t_0}^{t_{N}}$ to $\Phi_{t_1}^{t_{N+1}}$, given by
\begin{eqnarray*}
\Phi_{t_0}^{t_{N}} &=& \Phi_{t_{N-1}}^{t_N} \circ \cdots \circ \Phi_{t_1}^{t_2} \circ \Phi_{t_0}^{t_1} \\
\Phi_{t_1}^{t_{N+1}} &=& \Phi_{t_{N}}^{t_{N+1}} \circ \cdots \circ \Phi_{t_2}^{t_3} \circ \Phi_{t_1}^{t_2} \, ,
\end{eqnarray*}
are simply some interpolations. All these proposed improvements are only on the flow map computations, and therefore the FTLE. For the FSLE or the ISLE computations on one single time level and one particular $r$, say $\gamma_r(\x_{i,j},t_0)$, one has to go through all $M$ mesh points in the temporal direction. For a different time level or even a different separation factor $r$, we have to go though the same procedure all over again. We do not aware of many efficient algorithms for the FSLE or the ISLE implementation.

\subsection{An Eulerian method for computing the flow map}
\label{SubSec:oldEulerian}

\begin{figure}[!t]
\begin{center}
(a)\includegraphics[width=5.75cm]{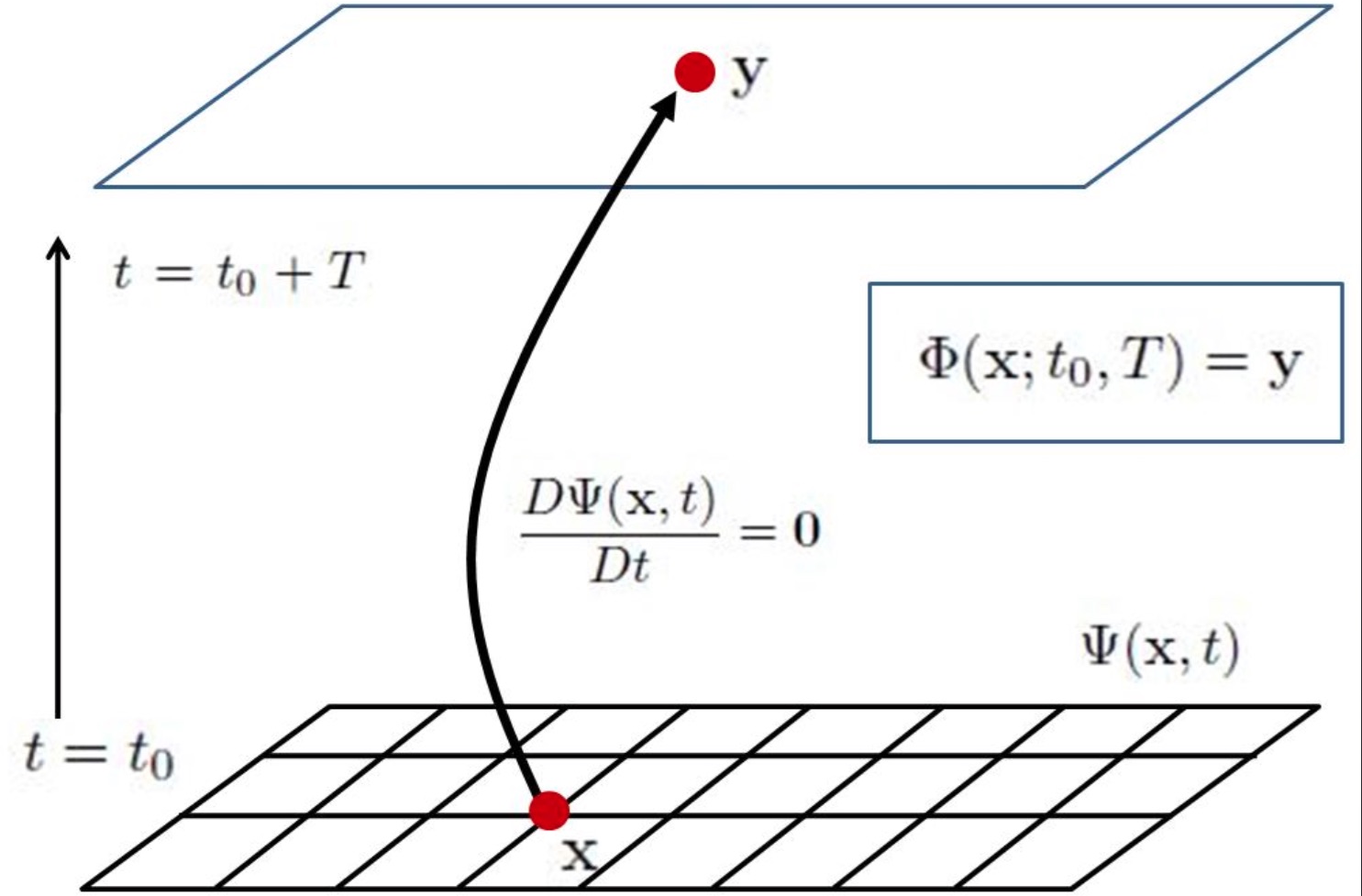}
(b)\includegraphics[width=5.75cm]{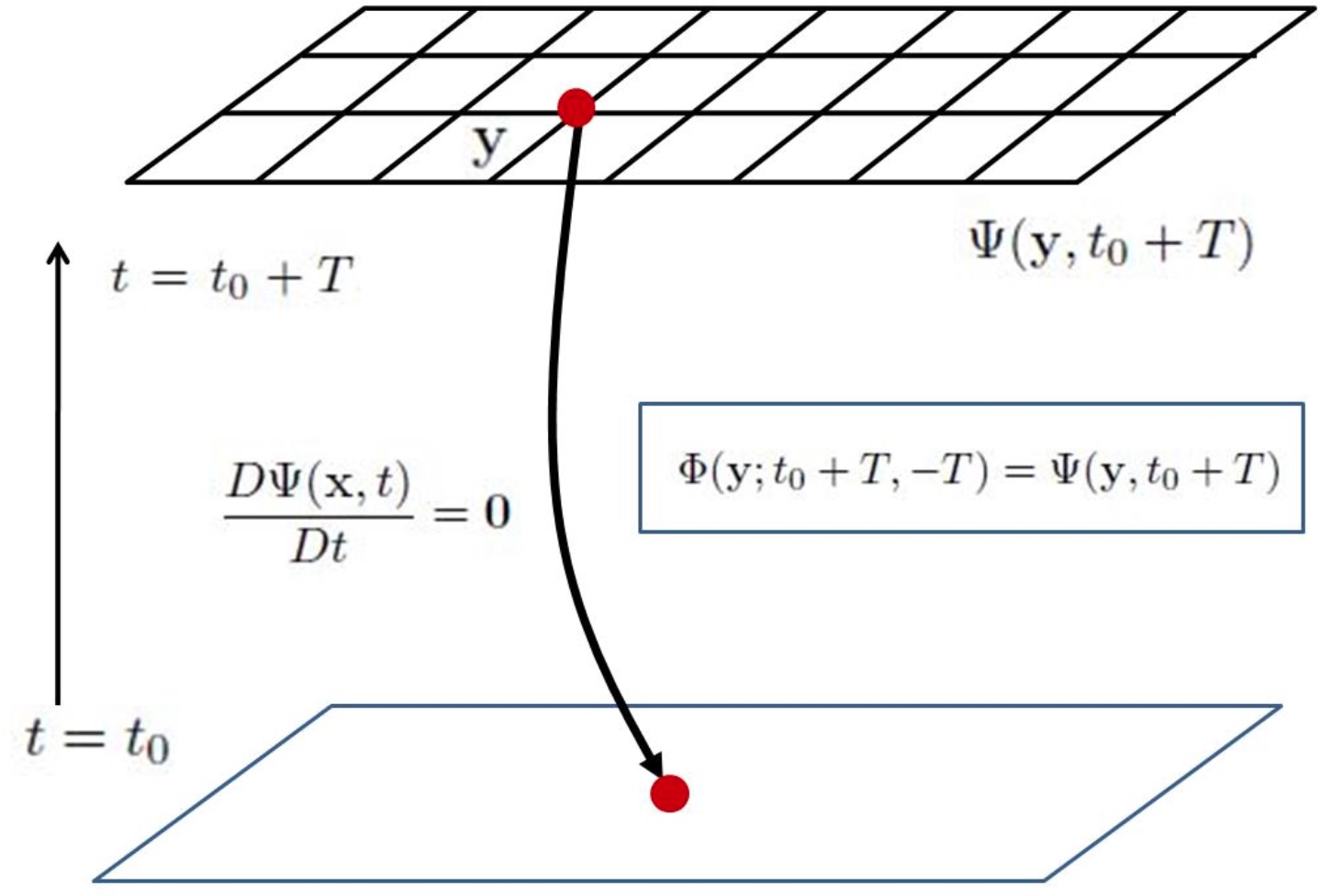}
\end{center}
\caption{Lagrangian and Eulerian interpretations of the function $\Psi$ \cite{leu11}. (a) Lagrangian ray tracing from a given grid location $\mathbf{x}$ at $t=0$. Note that $\mathbf{y}$ might be a non-grid point. (b) Eulerian values of $\Psi$ at a given grid location $\mathbf{y}$ at $t=T$ gives the corresponding take-off location at $t=0$. Note the take-off location might not be a mesh point.}
\label{Fig:ForwardBackward}
\end{figure}

We briefly summarize the Eulerian approach based on the level set method and the Louville equation. We refer interested readers to \cite{leu11} and thereafter. We define a vector-valued function {$\Psi=(\Psi^1,\Psi^2,\cdots,\Psi^d): \Omega \times \mathbb{R} \rightarrow \mathbb{R}^d$.} At $t=0$, we initialize these functions by
\begin{equation}
\Psi(\mathbf{x},0) = \mathbf{x} =(x^1,x^2,\cdots,x^d) \, .
\label{Eqn:LevelSetEquationIC}
\end{equation}
These functions provide a labeling for any particle in the phase space at $t=0$. In particular, any particle initially located at {$(\x,t)=(\mathbf{x}_0,0)=(x_0^1,x_0^2,\cdots,x_0^d,0)$} in the extended phase space can be \textbf{implicitly} represented by the intersection of $d$ {codimension-1 surfaces represented by $\cap_{i=1}^d \{\Psi^i(\mathbf{x},0)=x_0^i\}$ in $\mathbb{R}^d$}. Following the particle trajectory with $\mathbf{x}=\mathbf{x}_0$ as the initial condition in a
given velocity field, {any particle identity should be preserved in the Lagrangian framework and this implies that} the material derivative of these level set functions is zero, i.e.
\begin{equation*}
\frac{D \Psi(\mathbf{x},t)}{Dt} = \mathbf{0} \, .
\end{equation*}
This implies the following level set equations, or the Liouville equations,
\begin{equation}
\frac{\partial \Psi(\mathbf{x},t)}{\partial t} + (\mathbf{u} \cdot \nabla) \Psi(\mathbf{x},t) = \mathbf{0}
\label{Eqn:LevelSetEquation}
\end{equation}
with the initial condition (\ref{Eqn:LevelSetEquationIC}).

{The above \textbf{implicit} representation} embeds all path lines in the extended phase space. For instance, the trajectory of a particle initially located at $(\mathbf{x}_0,0)$ can be found by determining the intersection of $d$ {codimension-1 surfaces represented by $\cap_{i=1}^d \{\Psi^i(\mathbf{x},t)=x_0^i\}$} in the extended phase space. Furthermore, the forward flow map at {a grid location} $\mathbf{x}=\mathbf{x}_0$ from $t=0$ to $t=T$ is given by
$\Phi_0^T(\mathbf{x}_0) = \mathbf{y}$ {where} $\mathbf{y}$ satisfies $\Psi(\mathbf{y},0+T)=\Psi(\mathbf{x}_0,0) \equiv \mathbf{x}_0$. {Note that, in general, $\mathbf{y}$ is a non-mesh location. The typical two dimensional scenario is illustrated in Figure \ref{Fig:ForwardBackward} (a).}

The solution to (\ref{Eqn:LevelSetEquation}) contains much more information than what {was referred to above}. Consider a given mesh location $\mathbf{y}$ in the phase space at {the} time $t=T$, as shown in Figure \ref{Fig:ForwardBackward} (b), i.e. $(\mathbf{y},T)$ in the extended phase space. As discussed in our previous work, these level set functions $\Psi(\mathbf{y},T)$ {defined on a uniform Cartesian mesh} in fact give the {backward} flow map from $t=T$ to $t=0$, i.e.
$
\Phi_T^0(\mathbf{y})=\Psi(\mathbf{y},T)
$.
Moreover, {the} solution to the level set equations (\ref{Eqn:LevelSetEquation}) for $t\in (0,T)$ provides also backward flow maps for all intermediate times, i.e.
$
\Phi_t^0(\mathbf{y})=\Psi(\mathbf{y},t)
$.

To compute the forward flow map, on the other hand, \cite{leu11} has proposed to simply reverse the above process by initializing the level set functions at $t=T$ by
$
\Psi(\mathbf{x},T) = \mathbf{x}
$
and {solving} the corresponding level set equations (\ref{Eqn:LevelSetEquation}) {backward} in time. A typical algorithm of this type is given in {\sf Algorithm 1}. Note that in the original Lagrangian formulation, the {attracting} LCS or the {unstable} manifold is obtained by {backward} time tracing, while the {repelling} LCS or the {stable} manifold is computed by {forward} time integration. In this current Eulerian formulation, on the other hand, {forward} time integration of the Liouville equations gives the {attracting} LCS and {backward} time marching provides the {repelling} LCS.

\pic \\
\vspace{0.5cm} \noindent {\sf Algorithm 1: Computing the {forward} flow map $\Phi_0^T(\mathbf{x})$:
\begin{enumerate}
\item Discretize the computational domain to get $x_i, y_j,t_k$.
\item Initialize the level set functions on the last time level $t=t_N$ by
\begin{eqnarray*}
\Psi^1(x_i,y_j,t_k) &=& x_i \nonumber\\
\Psi^2(x_i,y_j,t_k) &=& y_j \, .
\end{eqnarray*}
\item Solve the Liouville equations for each individual level set function $l=1,2$
$$
\frac{\partial \Psi^l}{\partial t} + (\mathbf{u} \cdot \nabla) {\Psi}^l = 0
$$
from $t=t_{k}$ down to $t=0$ using any well-developed high order numerical methods like WENO5-TVDRK2 \cite{liuoshcha94,shu97,gotshu98} with the boundary conditions
\begin{eqnarray}
\Psi(\mathbf{x},t)|_{\mathbf{x} \in \partial \Omega} = \mathbf{x} \, && \mbox{ if $\mathbf{n} \cdot \mathbf{u} < 0$}
\label{Eqn:Inflow} \\
\mathbf{n} \cdot \nabla \Psi^l(\mathbf{x},t)|_{\mathbf{x} \in \partial \Omega} = 0 \, && \mbox{ if $\mathbf{n} \cdot \mathbf{u} > 0$}
\label{Eqn:Outflow}
\end{eqnarray}
where $\mathbf{n}$ is the outward normal of the boundary.
\item Assign $\Phi_0^T(x_i,y_j)=\Psi(x_i,y_j,0)$.
\end{enumerate}
}
\pic

\section{Our proposed approaches}
\label{Sec:Proposal}

In this section, we introduce two new Eulerian algorithms for flow visualizations. The first algorithm is to improve the Eulerian numerical method for computing the forward flow map. Instead of solving the PDE \textit{backward} in time as in Section \ref{SubSec:oldEulerian} and \cite{leu11}, we develop an Eulerian PDE algorithm to construct the forward flow map from $t=0$ to the final time $t=T$ so that we do not need to load the time-dependent velocity field from the terminal time level backward to the initial time. In other words, we construct the forward flow map \textit{on the fly}. This is computationally more natural and convenient.

\reminder{Because of this simple algorithm, we propose to develop an efficient and systematic PDE based method for computing the ISLE field.} Typical challenge in obtaining such field is that the choice of the separation factor $r$ is in general flow dependent. One has to vary $r$ to extract the necessary information. Usual Lagrangian methods unfortunately require shooting different sets of initial rays for each individual separation factor $r$. This is computationally very inefficient. The proposed Eulerian formulation requires either some simple interpolations or only one single level surface extraction which can be easily implemented using any well-developed contour extraction function such as \textsf{isosurface} in \textsf{MATLAB}.

\subsection{A forward time marching PDE approach for constructing the forward flow map}
\label{SubSec:newEulerian}
One disadvantage about the previous Eulerian approach for forward flow map computation as discussed in Section \ref{SubSec:oldEulerian} and \cite{leu11,leu13,youleu14} is that the level set equation
\begin{equation}
\frac{\partial \Psi(\mathbf{x},t)}{\partial t} + (\mathbf{u} \cdot \nabla) \Psi(\mathbf{x},t) = \mathbf{0}
\label{Eqn:LevelSet}
\end{equation}
has to be first solved \textit{backward} in time from $t=T$ to $t=0$. Once we have the final solution at the time level $t=0$, one can then identity the flow map $\Phi_0^T(\mathbf{x})$ by $\Psi(\mathbf{x},0)$. This could be inconvenient especially when we need to access the intermediate forward flow map. In this section, we propose a new algorithm to construct the forward flow map \textit{on the fly}.

Consider the forward flow map from $t=0$ to $t=T$. Suppose the time domain $[0,T]$ is discretized by $N+1$ discrete points $t_n$, where $t_0=0$ and $t_N=T$. Then on each subinterval $[t_n,t_{n+1}]$ for $n=0,1,\cdots,N-1$, we use the same method as described in Section \ref{SubSec:oldEulerian} to construct the forward flow map $\Phi_{t_n}^{t_{n+1}}$. In particular, we solve the level set equation (\ref{Eqn:LevelSet}) \textit{backward} in time from $t=t_{n+1}$ to $t=t_n$ with the terminal condition $\Psi(\x,t_{n+1})=\x$ imposed on the time level $t=t_{n+1}$. Then the forward flow map is given by $\Phi_{t_n}^{t_{n+1}}(\x)=\Psi(\x,t_{n})$. Once we have obtained this one step \textit{forward} flow map $\Phi_{t_n}^{t_{n+1}}$, the \textit{forward} flow map from $t=0$ to $t=T$ can be obtained using the composition $\Phi_0^{t_{n+1}}=\Phi_{t_n}^{t_{n+1}}\circ \Phi_0^{t_{n}}$. And this can be easily done by typical numerical interpolation methods. To prevent extrapolation, we could enforce $\Phi_{t_n}^{t_{n+1}} \subseteq [ x_{\min}, x_{\max} ] \times [y_{\min},y_{\max}]$.

Here we emphasize that even though we solve equation (\ref{Eqn:LevelSet}) \textit{backward} in time for each subinterval, we access the velocity field data $\mathbf{u}^n$ prior to $\mathbf{u}^{n+1}$ and therefore the forward flow map is indeed obtained \textit{on the fly}. \reminder{As a simple example, we consider the first order Euler method in the temporal direction. The forward flow map $\Phi_{t_n}^{t_{n+1}}$ can be obtained by first solving
$$
\frac{\Psi_{i,j}^{n}-\Psi_{i,j}^{n+1}}{\Delta t}-\mathbf{u}_{i,j}^{n+1}\cdot \nabla\Psi_{i,j}^{n+1} = \mathbf{0} 
$$
with the terminal condition $\Psi_{i,j}^{n+1}=\x_{i,j}$, and then assigning $\Phi_{t_n}^{t_{n+1}}=\Psi_{i,j}^{n}$. 
Higher order generalization is straight-forward.} For example, if we use the TVD-RK2 method in the temporal direction \cite{oshshu91}, we can compute the forward flow map $\Phi_{t_n}^{t_{n+1}}$ by
\begin{eqnarray*}
\frac{\reminder{\hat{\Psi}_{i,j}^{n}}-\Psi_{i,j}^{n+1}}{\Delta t}-\mathbf{u}_{i,j}^{n+1}\cdot \nabla\Psi_{i,j}^{n+1} &=& \mathbf{0} \\
\frac{\reminder{\hat{\Psi}_{i,j}^{n-1}}-\reminder{\hat{\Psi}_{i,j}^{n}}}{\Delta t}-\mathbf{u}_{i,j}^{n}\cdot \nabla \reminder{\hat{\Psi}_{i,j}^{n}} &=&\mathbf{0} \\
\Psi_{i,j}^n=\frac{1}{2} \left( \reminder{\hat{\Psi}_{i,j}^{n-1}}+\reminder{\Psi_{i,j}^{n+1}} \right)
\end{eqnarray*}
with the terminal condition $\Psi_{i,j}^{n+1}=\x_{i,j}$, \reminder{and the functions $\hat{\Psi}_{i,j}^{n-1}$ and $\hat{\Psi}_{i,j}^{n}$ are two intermediate predicted solutions at the time levels $t_{n-1}$ and $t_n$, respectively.} Then the flow map from the time level $t=t_n$ to $t=t_{n+1}$ is given by $\Phi_{t_n}^{t_{n+1}}(\x)=\Psi^n(\x)$.

It is rather natural to decompose the flow map $\Phi_{0}^{T}$ into a composition of maps
$\Phi_{0}^{T} = \Phi_{t_{N-1}}^{t_N} \circ \cdots \circ \Phi_{t_1}^{t_2} \circ \Phi_{0}^{t_1}$. For example, such idea has been used recently in \cite{brurow10} to improve the computational efficiency of the Lagrangian FTLE construction between various time levels. \reminder{Define the interpolation operator by $\mathcal{I}$, which returns the interpolated mesh values from the non-mesh values. Then the flow map $\Phi_{t_0}^{t_N}$ can be decomposed into
$
\Phi_{t_0}^{t_N} = \mathcal{I} \Phi_{t_{N-1}}^{t_N} \circ \cdots \circ \mathcal{I} \Phi_{t_1}^{t_2} \circ \mathcal{I} \Phi_{t_0}^{t_1}$.
If the maps $\Phi_{t_k}^{t_{k+1}}$ for $k=0,\cdots,N-2$ are all stored once they are computed, one can form $\Phi_{t_0}^{t_N}$ by determining only $\Phi_{t_{N-1}}^{t_N}$. } The idea in this work shares some similarities with what we are implementing. Our approach can also re-use all intermediate flow maps, the work \cite{brurow10} however has concentrated only on improving the computational efficiency of the Lagrangian approach. We are using the idea of flow map decomposition to propose a forward computational strategy for the Eulerian formulation.

To obtain a stable evolution in the flow map constructions, theoretically we require the interpolation scheme between two local flow maps to be monotone, i.e. the interpolation scheme should preserve the monotonicity of the given data points \cite{leu13}. This is especially important in the context of the phase flow maps for autonomous flows \cite{canyin06,leu13} since the overshooting due to the interpolation could be significantly amplified in the time doubling strategy
$$
\Phi_0^{t_{2^k}}=\left(\Phi_0^{t_{2^{k-1}}}\right)^2=\left[\left(\Phi_0^{t_{2^{k-2}}}\right)^2\right]^2 \, .
$$

\subsection{An application to ISLE computations}
\label{SubSec:ISLEComp}

In this section, we propose an efficient Eulerian approach to compute the ISLE function for an arbitrary separation factor $r$ based on the techniques developed in the previous section.

To compute the ISLE at $\x$, one has to determine the minimum time $\tau_r(\x)$ at each location $\x$ for which $\sqrt{\lambda_0^{\tau_r(\x)}(\x)}=r$. It is, therefore, necessary to keep track of all intermediate values $\lambda_0^{t_n}(\x)$ for all $n$'s during the flow map construction. At each grid point $\x_{i,j}$ and each intermediate time step $t_n$, we construct a quantity $s_0^{t_n}(\x_{i,j})$ by
$$
s_0^{t_n}(\x_{i,j}) = \max\left[\sqrt{\lambda_0^{t_n}(\x_{i,j})},s_0^{t_{n-1}}(\x_{i,j})\right]
$$
with $s_0^0(\x_{i,j})=0$. Once we determine the flow map $\Phi_0^{t_N}$, we have also constructed an increasing sequence at each grid point $\x_{i,j}$ in time $\{ s_0^{t_n}(\x_{i,j}) : n=0,1,\dots, N \}$, i.e.
$$
0=s_0^0(\x_{i,j}) \le s_0^{t_1}(\x_{i,j}) \le \cdots \le s_0^{t_N}(\x_{i,j})
$$
for each $\x_{i,j}$. Now, we interpret this quantity as a spatial-temporal volumetric data in the $\x-t$ space. Due to such monotonicity property in the temporal direction, the isosurface
$$
\left\{ \left(\x,t^* \right) : s_0^{t^*}(\x)=r \right\}
$$
forms a graph of $\x$, i.e. for each $\x$ we have a unique $t^*=t^*(\x)$ so that $s_0^{t^*(\x)}(\x)=r$. And more importantly, the value $t^*(\x)$ gives the shortest time so that $\sqrt{\lambda_0^{t^*(\x)}(\x)}=r$. This implies that $\tau_r(\x_{i,j})=t^*(\x_{i,j})$ and so the level surface in the extended phase space, i.e.
$$
\left\{ \left(\x,\tau_r^*(\x) \right) : s_0^{\tau_r^*(\x)}(\x)=r \right\} \, ,
$$
can be used to define $\tau_r(\x_{i,j})$. Finally, the ISLE function can then be computed using
$$
\gamma_r(\x_{i,j},0)=\frac{1}{\left| \tau_r(\x_{i,j}) \right|} \ln r \, .
$$
The computational algorithm is summarized in {\sf Algorithm 2}.

\vspace{0.25cm}
\noindent
\pic\\
{\sf Algorithm 2 (Our Proposed Eulerian Method for Computing the ISLE):
\begin{enumerate}
\item Discretize the computational domain
\begin{eqnarray*}
  &&  x_i=x_{\min}+(i-1)\Delta x,\quad \Delta x=\frac{x_{\max}-x_{\min}}{I-1},\quad i=1,2,\cdots,I, \\
  && y_j=y_{\min}+(j-1)\Delta y,\quad \Delta y=\frac{y_{\max}-y_{\min}}{J-1},\quad j=1,2,\cdots,J, \\
  && t_n=n\Delta t,\quad \Delta t=\frac{T}{N},\quad n=0,1,...,N \, .
\end{eqnarray*}
\item Set $s_0^0(\x)=0$.
\item For $n=0,1,\cdots,N-1$,
\begin{enumerate}
\item Solve the level set equation (\ref{Eqn:LevelSet}) for $\Psi_{i,j}^n$ with the condition $\Psi_{i,j}^{n+1}=(x_i,y_j)$.
\item Define $\Phi_{t_n}^{t_{n+1}}=\Psi_{i,j}^n$.
\item Interpolate the solutions to obtain $\Phi_0^{t_{n+1}}=\Phi_{t_n}^{t_{n+1}}\circ \Phi_0^{t_n}$.
\item Compute $\sqrt{\lambda_0^{t_{n+1}}(\x_{i,j})}$ and determine
$$
s_0^{t_{n+1}}(\x_{i,j})=\max \left[ \sqrt{\lambda_0^{t_{n+1}}(\x_{i,j})},s_0^{t_n}(\x_{i,j})\right] \, .
$$
\end{enumerate}
\item  Given an arbitrary separation factor $r$. For each $\x_{i,j}$, search for a time step $t_n$ such that $s_0^{t_n}(\x_{i,j}) \leq r \leq s_0^{t_{n+1}}(\x_{i,j})$. Then determine $\tau_r(\x_{i,j})$ by applying the piecewise linear interpolation using the values $s_0^{t_n}(\x_{i,j})$ and $s_0^{t_{n+1}}(\x_{i,j})$.

The ISLE function is computed as
$$
\gamma_r(\x_{i,j},0)=\frac{1}{\left| \tau_r(\x_{i,j}) \right|}\ln r.
$$
\end{enumerate}

}
\pic \\

There are several major advantages of the proposed Eulerian approach. It provides an efficient yet simple numerical implementation of $\gamma_r(\x,0)$ for multiple separation factors $r$. The construction of the quantity $s_0^t$ can be first done in the background. For each particular given value of $r$, one has to perform only one single isosurface extraction in the extended phase space, i.e. the $\x-t$ space. Unlike the typical Lagrangian implementation where one has to repeatedly keep track of {\it all} particle trajectories for the constraint $s_0^t=r$, the proposed Eulerian algorithm provides a systematic way based on a simple thresholding strategy.

\begin{remark}
Note that in our implementation, if it happens that $s_0^{t_n}(x_i,y_j)< r$ for all $t_n$, we simply set $\gamma_r(x_i,y_j,0)=0$.
\end{remark}

\subsection{Computational complexities}

We conclude this section by discussing the computational complexity of our proposed algorithm for a two dimensional flow, i.e. $d=2$. Let $N$ and $M$ be the discretization size of one spatial dimension and time dimension respectively. Since the Liouville equation is a hyperbolic equation, we have $M=O(N)$ by the CFL condition. At each time step $t_n$, a short time flow map $\Phi_{t_n}^{t_{n+1}}$ is obtained by solving the Liouville equations from $t_{n+1}$ to $t_{n}$, the computational effort is $O(N^2)$. Then in computing the long time flow map $\Phi_0^{t_{n+1}}$, an interpolation $\Phi_0^{t_{n+1}} = \Phi_{t_n}^{t_{n+1}} \circ \Phi_0^{t_n}$ takes $O(N^2)$ operations. Therefore the construction of flow map at each time step requires $O(N^2)+O(N^2)=O(N^2)$ operations. Computing $\sqrt{\lambda_0^{t_n}(\x_{i,j})}$ and $s_0^{t_{n+1}}(\x_{i,j})$ also needs $O(N^2)$ operations, the complexity order is kept at $O(N^2)$. Summing up this procedure in all time steps, the overall computational complexity is $M \cdot O(N^2)=O(N^3)$. 

\reminder{This new Eulerian method has significantly improved the overall computational complexity in the application to the ISLE computations. Since the calculations involve the forward flow map from $t=t_0$ to $t=t_i$ for all $i=1,2,\cdots,M$, the original Eulerian approach as discussed in Section \ref{SubSec:oldEulerian} requires to solve the Liouville equation from each individual time level $t=t_i$ backward in time to $t=t_0$. The overall computational complexity in obtaining all flow maps is therefore given by $O(M^2N^2)=O(N^4)$ which is one order magnitude larger than the newly developed Eulerian method.}

\section{A relationship between the FTLE and the ISLE}
\label{Sec:Relation}

\reminder{With algorithms proposed in Section \ref{Sec:Proposal}, we can finally develop an efficient numerical approach to compute the \textit{forward} FTLE field by solving corresponding PDEs \textit{forward} in time and the ISLE field for arbitrary separation factor $r$.} The FTLE and the ISLE (or the FSLE) are indeed measuring different properties of a given flow, even though both quantities depend on the same function ${\lambda_0^t(\x)}$ which relates the growth of an infinitesimal perturbation. In particular, FTLE measures how much ${\lambda_0^t(\x)}$ grows over a finite time span $[0,t]$, while the ISLE measures the time required for the quantity ${\lambda_0^t(\x)}$ to reach a threshold value prescribed by the separation factor. They are two different tools to study how rapid ${\lambda_0^t(\x)}$ grows. A careful comparison of these two quantities can be found in, for example, \cite{ppss14}.

\reminder{In many numerical experiences, on the other hand, the FTLE fields can show striking visual resemblance with the ISLE fields of certain separation factors. However, there has not been much studies on the theoretical relationship between these two quantities.} In this section, we are going to demonstrate how these two quantities relate to each other by considering their corresponding {\it ridges}. Of course, the way to define a {\it ridge} is not unique and one can pick a convenient definition for a particular application \cite{spft11,allpea15}. In this work, we adapt the following simple definition.

\begin{defn}
For a given scalar function $f(\x)$, we define a $f$-ridge to be a codimension-one compact and $C^1$ surface, denoted by $\mathcal{M}$, such that for every point $\x \in \mathcal{M}$, $f(\x)$ is a local maximum along the normal direction to $\mathcal{M}$ at $\x$.
\end{defn}


At these locations, $f(\x)$ is also called a generalized maximum in some literatures. Next, we define a particular neighbourhood of a ridge for later discussion.

\begin{defn}
Suppose $\mathcal{S}$ is a codimension-one, compact and smooth surface. We define $\Gamma_\mathcal{S}(\rho)$, the tubular neighbourhood of $\mathcal{S}$ with radius $\rho>0$, as the collection of any point $\y$ that can be uniquely expressed as $\x+s \n_{\x}$, for some point $\x \in \mathcal{S}$, a real number $s \in (-\rho,\rho)$, and $\n_{\x}$ is an unit normal vector of $\mathcal{S}$ at $\x$.
\end{defn}

\begin{remark}
To ensure that any point $\y$ on $\Gamma_\mathcal{S}(\rho)$ can be uniquely expressed as such, one can equivalently require that the boundary of $\Gamma_\mathcal{S}(\rho)$ does not intersect itself. For a smooth surface, this can always be guaranteed by choosing a small enough positive number $\rho$.
\end{remark}

The definitions of both the FTLE and the ISLE involve $\lambda_0^t(\x)$. It is therefore natural to link them through this particular quantity. The following simple \reminder{theorem} which identifies the $\sigma_0^t$-ridges with the $\lambda_0^t$-ridges, will be the very first step in building up our analysis.

\begin{thm}
\label{prop:FTLE&lambda}
Let $\sigma_0^t$ and $\lambda_0^t$ be the FTLE function and the largest eigenvalue of the Cauchy-Green deformation tensor from time $0$ to time $t$, respectively. Then, a $\sigma_0^t$-ridge is also a $\lambda_0^t$-ridge.
\end{thm}

\begin{proof}
Since $(2|t|)^{-1}$ is a just a scaling factor and the natural logarithm function is strictly increasing, all generalized maxima of $\sigma_0^t$ as defined in (\ref{Eqn:FTLE}) and $\lambda_0^t$ coincide.
\end{proof}

For convenience, \reminder{in the time window $[0,t]$, we impose $\gamma_r(\x,0)=0$ if $\sqrt{\lambda_0^s(\mathbf{x})}<r$ for all $s \in [0,t]$}. If an ISLE function $\gamma_r(\x,0) \equiv 0$ on an open set $\Omega$ of the spatial domain, there is obviously no visible ISLE feature on $\Omega$, corresponding to this separation factor $r$. The following \reminder{theorem} establishes a fundamental result that for certain choices of $r$, the ISLE function $\gamma_r$ can be positive on a tubular neighbourhood of a $\lambda_0^t$-ridge. This explains the non-trivial ISLE features around the location of a FTLE ridge, when one chooses a suitable separation factor $r$.

\begin{thm}
\label{prop:funProp1}
Let $\lambda_0^t$ be the largest eigenvalue of the Cauchy-Green deformation tensor from time $0$ to time $t$. Suppose $\lambda_0^t$ is continuous everywhere in the spatial domain, and $\mathcal{M}$ is a $\lambda_0^t$-ridge with $m=\min\limits_{\x \in \mathcal{M}}\sqrt{\lambda_0^t(\x)} > 0$. Then for any positive number $r \in (0,m)$, there is a tubular neighbourhood $\Gamma_\mathcal{M}(\rho)$ such that the ISLE function is positive on $\Gamma_\mathcal{M}(\rho)$.
\end{thm}

\begin{proof}
For each $\x \in \mathcal{M}$, using the definition of $m$, we have $\sqrt{\lambda_0^t(\x)} \geq m > r$. Then by the continuity of $\lambda_0^t$ at $\x$, we can find an open ball  $B(\x,\rho_\x)$ with center $\x$ and radius $\rho_\x$ such that $\sqrt{\lambda_0^t(\y)} \geq r$ for all $\y \in B(\x,\rho_\x)$. Collect all these open balls at every point on $\mathcal{M}$ into  $\mathcal{F}$, which is clearly an open cover of $\mathcal{M}$. By the compactness of $\mathcal{M}$, we can find a finite subcover $\mathcal{F}'=\{B(\x_1,\rho_{\x_1}),\dots,B(\x_n,\rho_{\x_n})\}$ from $\mathcal{F}$. Take $\rho=\min\limits_{1 \leq i \leq n}\rho_{\x_i}$. Without loss of generality, assume $\rho$ is small enough such that the tubular neighbourhood $\Gamma_\mathcal{M}(\rho)$ is contained in the union of sets in $\mathcal{F}'$. Now for any $\y \in \Gamma_\mathcal{M}(\rho)$, it is clear that $\sqrt{\lambda_0^{t}(\y)} \geq r$. \reminder{Because the function $\sqrt{\lambda_0^{\tau}(\y)}$ is a continuous function of $\tau$, we apply the Intermediate Value Theorem and conclude that it} must have reached $r$ at some $\tau=\tau_r(\y) \in (0,t]$. Therefore $\gamma_r(\y,0)$ is positive.
\end{proof}

The next \reminder{theorem} is to identify a $\lambda_0^t$-ridge with a $\gamma_r$-ridge of a certain ISLE function $\gamma_r$.  To justify the condition, the points on the location of a $\lambda_0^t$-ridge, are expected to exhibit highly chaotic behavior. It is reasonable to assume that we can find a tubular neighbourhood of a $\lambda_0^t$-ridge, where the time monotonicity of $\lambda_0^{\tau}$ holds for $\tau \in [0,t]$, and that no other points reach the minimum value of $\lambda_0$, except those on $\mathcal{M}$.

\begin{thm}
\label{prop:funProp2}
Let $\lambda_0^t$ be the largest eigenvalue of the Cauchy-Green deformation tensor from time $0$ to time $t$. Suppose $\mathcal{M}$ is a $\lambda_0^t$-ridge with $m=\min\limits_{\x \in \mathcal{M}}\sqrt{\lambda_0^t(\x)}>0$. If there exists a tubular neighbourhood $\Gamma_\mathcal{M}(\rho)$ of $\mathcal{M}$ such that
\begin{enumerate}
\item for any fixed $\y \in \Gamma_\mathcal{M}(\rho)$, the quantity $\lambda_0^{\tau}(\y)$ is increasing in $\tau \in [0,t]$, and
\item $\{\x \in \Gamma_\mathcal{M}(\rho) : \lambda_0^t(\x) \geq m^2 \} = \mathcal{M}$,
\end{enumerate}
then $\mathcal{M}$ is also a $\gamma_m$-ridge for the ISLE function $\gamma_m$.
\end{thm}

\begin{proof}
Let $\x \in \mathcal{M}$, and $\y \in \Gamma_\mathcal{M}(\rho)-\{\x\}$ be a point in the normal direction of $\mathcal{M}$ at $\x$. By the second condition, we have $\sqrt{\lambda_0^t(\y)} < m$. And since $\lambda_0^{\tau}(\y)$ is increasing with $\tau$, it follows that $\sqrt{\lambda_0^{\tau}(\y)} <m$ for $\tau \in [0,t]$. Therefore $\gamma_m(\y,0)=0$ by our convention. We can conclude that no other points on the normal direction of $\mathcal{M}$ at $\x$ has positive value of $\gamma_m$, except $\x$ itself. This makes $\x$ a generalized maximum of $\gamma_m$.
\end{proof}

\reminder{Theorem} \ref{prop:funProp2} is a restricted result, in the sense that the whole $\lambda_0^t$-ridge is preserved for a particular $\gamma_r$ function. But it is straightforward to extend it to the following corollary, which looks for a certain portion of a $\lambda_0^t$-ridge that can be conserved in a larger class of ISLE functions.

\begin{cor}
\label{prop:funCor}
Let $\lambda_0^t$ be the largest eigenvalue of the Cauchy-Green deformation tensor from time $0$ to time $t$. Suppose $\mathcal{M}$ is a $\lambda_0^t$-ridge with $m=\min\limits_{\x \in \mathcal{M}}\sqrt{\lambda_0^t(\x)}>0$, and $\mathcal{M}'$ is a connected subset of $\mathcal{M}$ with $m'=\min\limits_{\x \in \mathcal{M}'}\sqrt{\lambda_0^t(\x)}>0$. For any positive number $r \in [m,m']$, if there exists a tubular neighbourhood $\Gamma_{\mathcal{M}'}(\rho)$ of $\mathcal{M}'$ such that
\begin{enumerate}
\item for any fixed $\y \in \Gamma_{\mathcal{M}'}(\rho)$, the quantity $\lambda_0^{\tau}(\y)$ is increasing in $\tau \in [0,t]$, and
\item $\{\x \in \Gamma_{\mathcal{M}'}(\rho) : \lambda_0^t(\x) \geq r^2 \} = \mathcal{M}'$,
\end{enumerate}
then $\mathcal{M}'$ is a $\gamma_r$-ridge for the ISLE function $\gamma_r$.
\end{cor}

\begin{proof}
The proof is parallel to that of \reminder{Theorem} \ref{prop:funProp2}, with $\mathcal{M}$ and $m$ being substituted by $\mathcal{M}'$ and $r$, respectively.
\end{proof}


Through \reminder{Theorem} \ref{prop:funProp1} to Corollary \ref{prop:funCor}, we have obtained a close relationship between the $\lambda_0^t$-ridges and the ISLE ridges. By the equivalence of the FTLE ridges and the $\lambda_0^t$-ridges as proved in \reminder{Theorem} \ref{prop:FTLE&lambda}, we can then link the FTLE ridges and the ISLE ridges. In practice, the estimated location of the ISLE ridges can be inferred by the location of the FTLE ridges, while $\lambda_0^t$ plays an important role for estimating a suitable separation factor defined in the ISLE ridge. To demonstrate this, suppose that there is a FTLE ridge $\mathcal{M}$ satisfying all conditions in \reminder{Theorem} \ref{prop:funProp2}. Then $\mathcal{M}$ is an ISLE ridge with the separation factor $m=\min\limits_{\x \in \mathcal{M}}\sqrt{\lambda_0^t(\x)}$. Although the value of $m$ is actually an unknown, we can determine a number $l$ that is slightly less than $\min\limits_{\x \in \mathcal{M}}\sigma_0^t(\x)$ manually. From the definition of the FTLE, we have
$$
m=\min\limits_{\x \in \mathcal{M}}\sqrt{\lambda_0^t(\x)} = \exp\left[{t\min\limits_{\x \in \mathcal{M}}\sigma_0^t(\x)}\right] > e^{lt} \, .
$$

Now, with any separation factor $r\in[e^{lt},m)$, \reminder{Theorem} \ref{prop:funProp1} implies that there is a tubular neighborhood around $\mathcal{M}$ with positive ISLE values. We can then increase the separation factor from $e^{lt}$ approaching to $m$, so that the tubular neighborhood containing $\mathcal{M}$ becomes narrower, and thus leads to a better approximation of $\mathcal{M}$. \reminder{In practice, the value of $l$ has to be chosen in a trial-by-error fashion and has to be determined case-by-case. Based on the Eulerian algorithm we developed in Section \ref{SubSec:newEulerian} and Section \ref{SubSec:ISLEComp}, we are now able to efficiently construct the ISLE field of each individual separation factor using one single isocontour extraction.}

\section{Numerical examples}
\label{Sec:Examples}

\reminder{In this section, we present numerical results to show the feasibility of the proposed Eulerian approach. We will compare the solution from the new approach with the previous Eulerian approach discussed in Section \ref{SubSec:oldEulerian}. It is noteworthy that for a dynamical system in a fixed time window, we can now construct the ISLE function with an arbitrary choice of the separation factor in only one single numerical flow simulation. Therefore, we can easily verify the theoretical connection between FTLE and ISLE, developed in Section \ref{Sec:Relation}.}

\subsection{The double gyre flow}
\label{SubSec:Example_Double_Gyre}

This first example is a simple flow taken from \cite{shalekmar05} to describe a periodically varying double-gyre. The flow is modeled by the following stream-function
$$
\psi(x,y,t)=A \sin[ \pi g(x,t) ] \sin(\pi y) \, ,
$$
where
\begin{eqnarray*}
g(x,t) &=& a(t) x^2 + b(t) x \, , \nonumber\\
a(t) &=& \epsilon \sin(\omega t) \, , \nonumber\\
b(t) &=& 1- 2\epsilon \sin(\omega t) \, .
\end{eqnarray*}
In this example, we follow \cite{shalekmar05} and use $A=0.1$, $\omega=2\pi/10$. We discretize the domain $[0,2]\times[0,1]$ using 513 grid points in the $x$-direction and 257 grid points in the $y$-direction. This gives $\Delta x=\Delta y=1/256$.

\begin{figure}[!htb]
\centering{
\includegraphics[width=12cm]{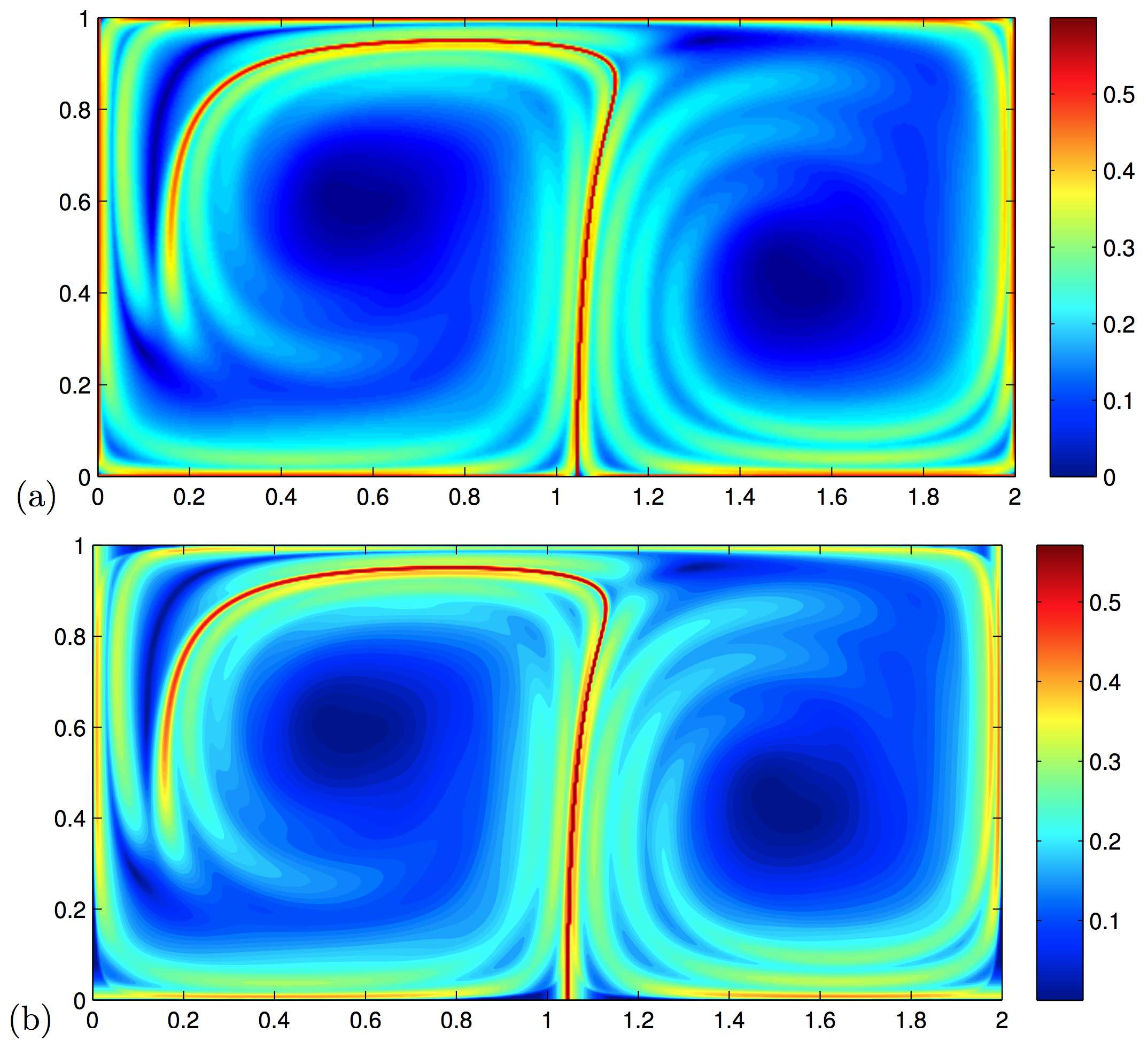}
}
\caption{(Section \ref{SubSec:Example_Double_Gyre}) The FTLE field $\sigma_0^{10}(\mathbf{x})$ computed using (a) the Lagrangian approach and (b) our proposed Eulerian approach in Section \ref{SubSec:newEulerian}.}
\label{Fig:FTLE}
\end{figure}

\begin{figure}[!htb]
\centering{
\includegraphics[width=8.5cm]{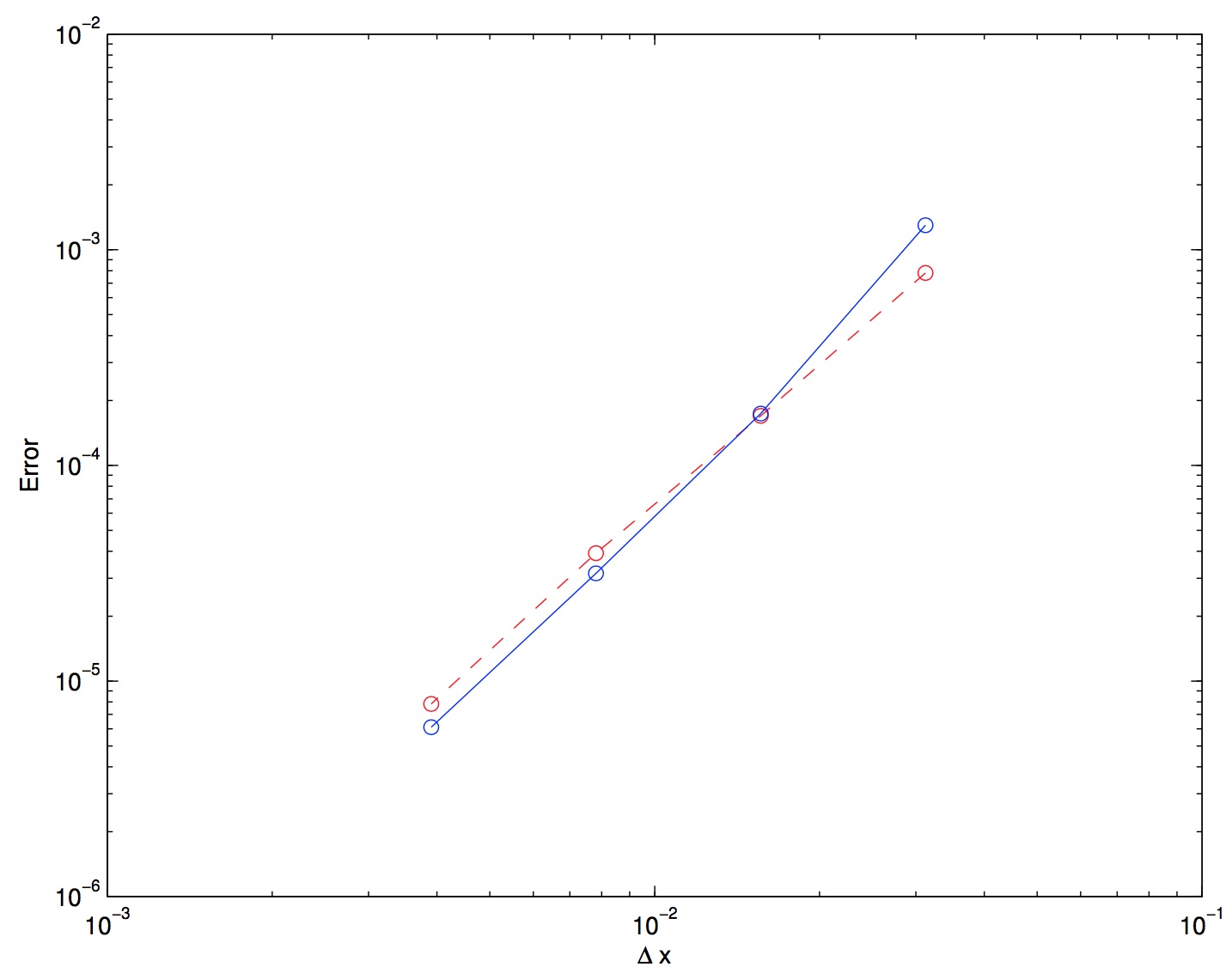}
}
\caption{\reminder{(Section \ref{SubSec:Example_Double_Gyre}) $L_2$-errors in both $\phi(x,y)$ (blue solid line) and $\psi(x,y)$ (red dash line) are computed using different $\Delta x$'s. The slope of these two lines is approximately 2.}}
\label{Fig:convergence}
\end{figure}

\begin{figure}[!htb]
\centering{
\includegraphics[width=12cm]{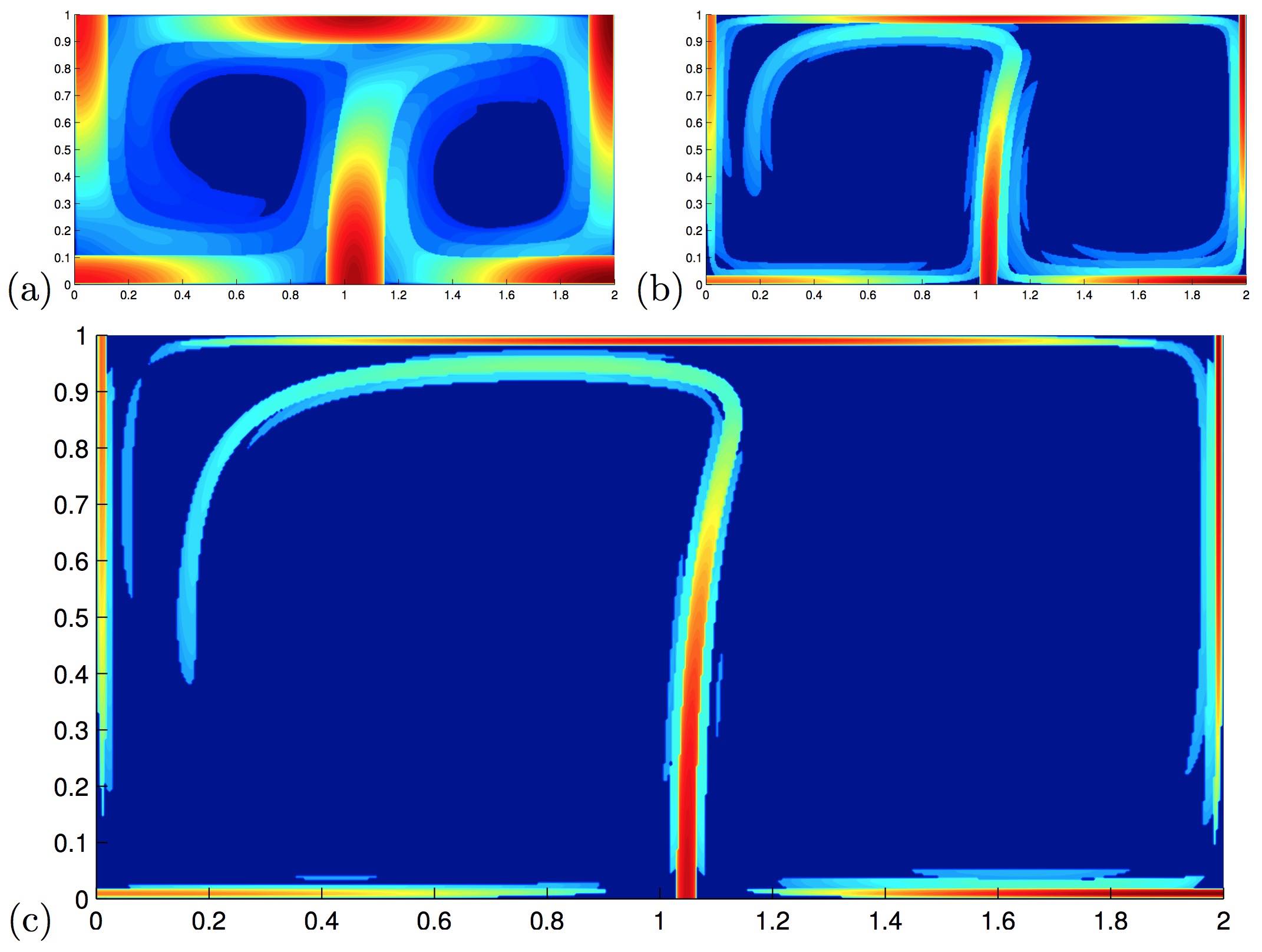}
}
\caption{(Section \ref{SubSec:Example_Double_Gyre}) The ISLE fields $\gamma_r(\mathbf{x},0)$ computed using our proposed Eulerian approach in Section \ref{SubSec:newEulerian} with (a) $r=3$, (b) $r=10$ and (c) $r=20$. It is clear that the FTLE ridges in Figure \ref{Fig:FTLE} are located inside some narrow regions of the ISLE ridges.}
\label{Fig:ISLE1}
\end{figure}

\begin{figure}[!htb]
\centering{
\includegraphics[width=12cm]{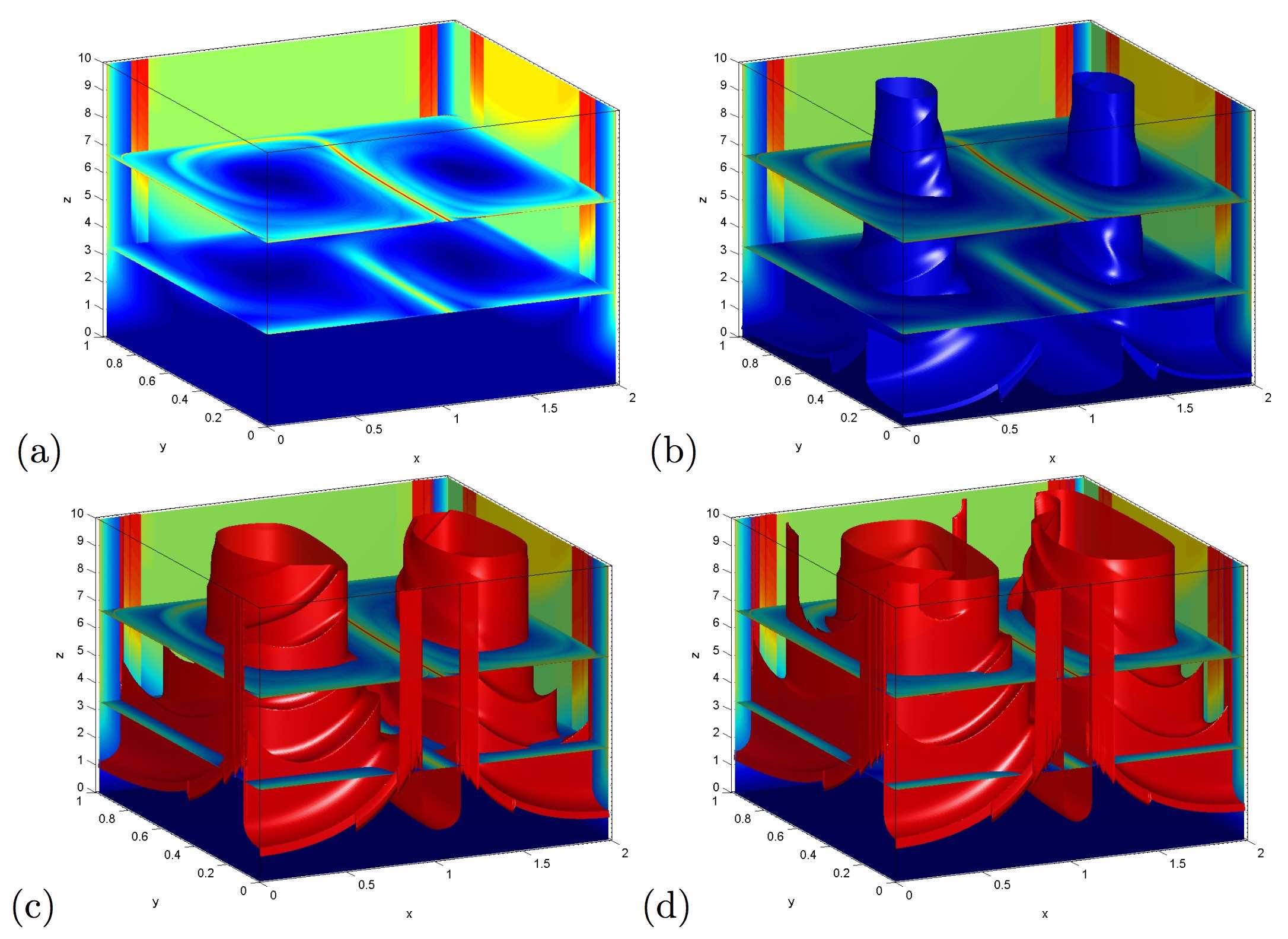}
}
\caption{(Section \ref{SubSec:Example_Double_Gyre} with $\epsilon=0.1$)
(a) The FTLE field $\sigma_0^t$ for $t=\frac{10}{3}$ and $t=\frac{20}{3}$.
The isosurfaces of the ISLE field $\gamma_r(\mathbf{x},0)$ with (b) $r$=1.5, (c) $r$=3 and (d) $r$=4. }
\label{Fig:Isosurface}
\end{figure}

In Figure \ref{Fig:FTLE}, we compare the FTLE field $\sigma_0^{10}(\mathbf{x})$ of the double-gyre flow with $\epsilon=0.1$ computed using the Lagrangian approach and our proposed Eulerian approach, respectively. These two solutions match extremely well. \reminder{We have also checked the $L_2$-error of the forward flow map $\Phi_0^{10}(x,y)$ using different $\Delta x$'s ranging from $1/32$ to $1/256$. Since we do not have the exact solution for this flow, the \textit{exact} solutions are computed using the Lagrangian ray tracing with a very small time step. Figure \ref{Fig:convergence} shows the errors in the flow map $\Phi_0^{10} : (x,y) \rightarrow \left(\phi(x,y),\psi(x,y) \right)$ with respect to different $\Delta x$'s. We find that the flow map computed using the proposed Eulerian approach is approximately second order accurate.
}

Figures \ref{Fig:ISLE1} shows the ISLE fields $\gamma_r(\mathbf{x},0)$ with several separation factors $r$. In the implementations, the flow map is computed from $t=0$ up to $t=10$. Recall that we impose zero ISLE value at points whose local separation rate is less than the given separation factor. With the increasing separation factor, the regions with non-zero ISLE values, get smaller. Nevertheless, the regions with non-zero or high ISLE values, always concentrate near the location of the FTLE ridges shown in Figure \ref{Fig:FTLE}. We denote by $\mathcal{M}$ the most prominent ridge originating from the central bottom of the domain. Note that a rough estimate of the minimum FTLE value on $\mathcal{M}$ is approximately 0.3. And so we have the estimate $\min\limits_{\x \in \mathcal{M}} \sqrt{\lambda_0^t(\x)} \geq e^{0.3 \times 10}= e^3 > 20$. Using \reminder{Theorem} \ref{prop:funProp1}, we can always find a tubular neighborhood near $\mathcal{M}$ of non-zero ISLE using a separation factor $r$ smaller than approximately 20. In Figure \ref{Fig:ISLE1}, we have computed the ISLE using various separation factors from $r=3$ to $r=20$. As $r$ increases, we can see that there are indeed non-zero ISLE regions around the major ridge $\mathcal{M}$. And evidently, the ISLE ridges lie within these regions.

To better compare the FTLE field and the ISLE field, we plot these two solutions together in Figure \ref{Fig:Isosurface}. The $x$- and the $y$- axis represent the computational domain $\Omega$ while the $z$ axis denotes the temporal direction. In (a), we plot the FTLE field $\sigma_0^{10/3}(\mathbf{x})$ at $t=10/3$ and $\sigma_0^{20/3}(\mathbf{x})$ at $t=20/3$. In (b-d), we plot the isosurfaces of the ISLE field $\lambda_r(\mathbf{x},0)$ with $r=1.5,3,4$, respectively. The $z$-values denote the time required for the local separation ratio to achieve $r$ for the first time, i.e. the $z$-values are the $\tau_r(\mathbf{x})$ we defined before. Near the major FTLE ridge, we can see the values of $\tau_r(\x)$ are significantly lower than the remaining part of the domain. Therefore the corresponding ISLE values are generalized maxima. This shows us the ISLE ridges match with the FTLE ridges very well.


\subsection{A simple analytic field}
\label{SubSec:Example_Analytic}

\begin{figure}[!htb]
\centering{
\includegraphics[width=12cm]{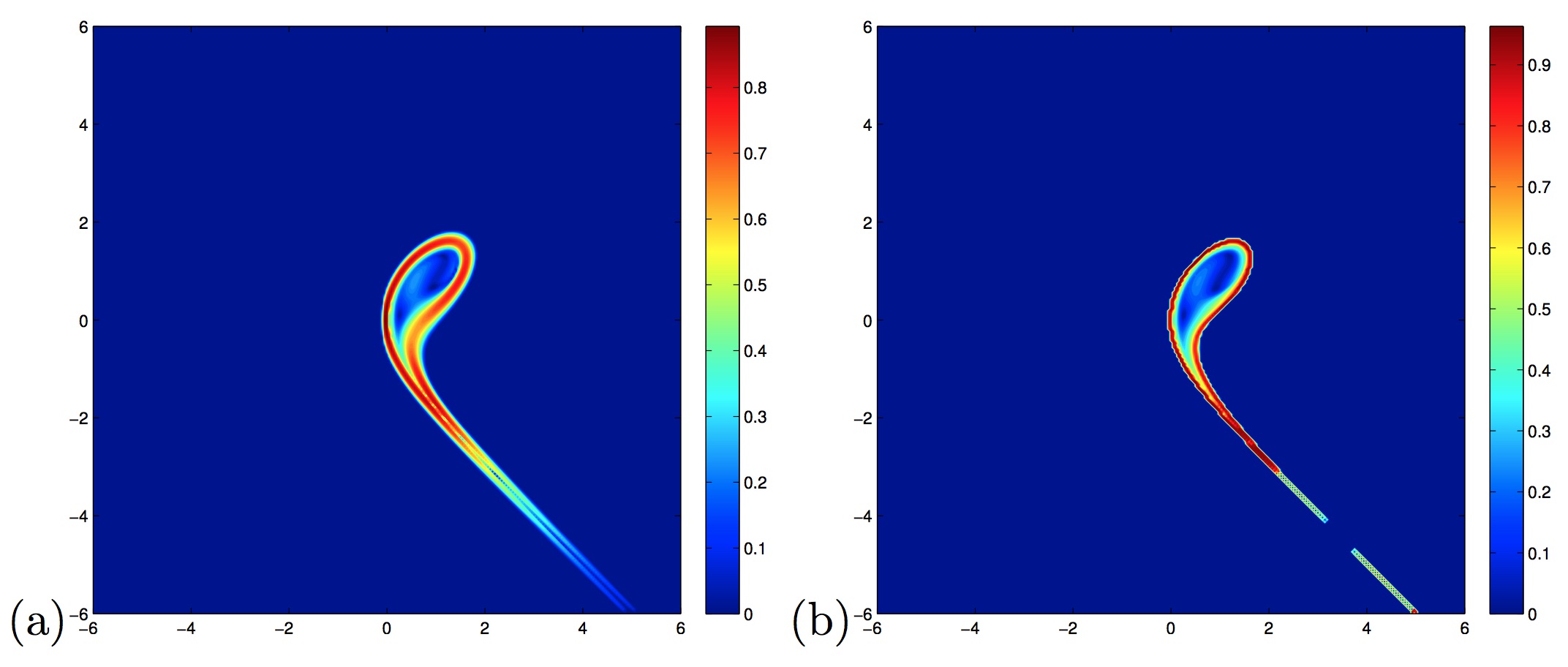}
}
\caption{(Section \ref{SubSec:Example_Analytic}) The FTLE field $\sigma_0^5(\mathbf{x})$ computed using (a) the original Eulerian approach \reminder{in Section \ref{SubSec:oldEulerian}}, and (b) our proposed Eulerian approach in Section \ref{SubSec:newEulerian}.}
\label{Fig:FTLE2}
\end{figure}

\begin{figure}[!htb]
\centering{
\includegraphics[width=12cm]{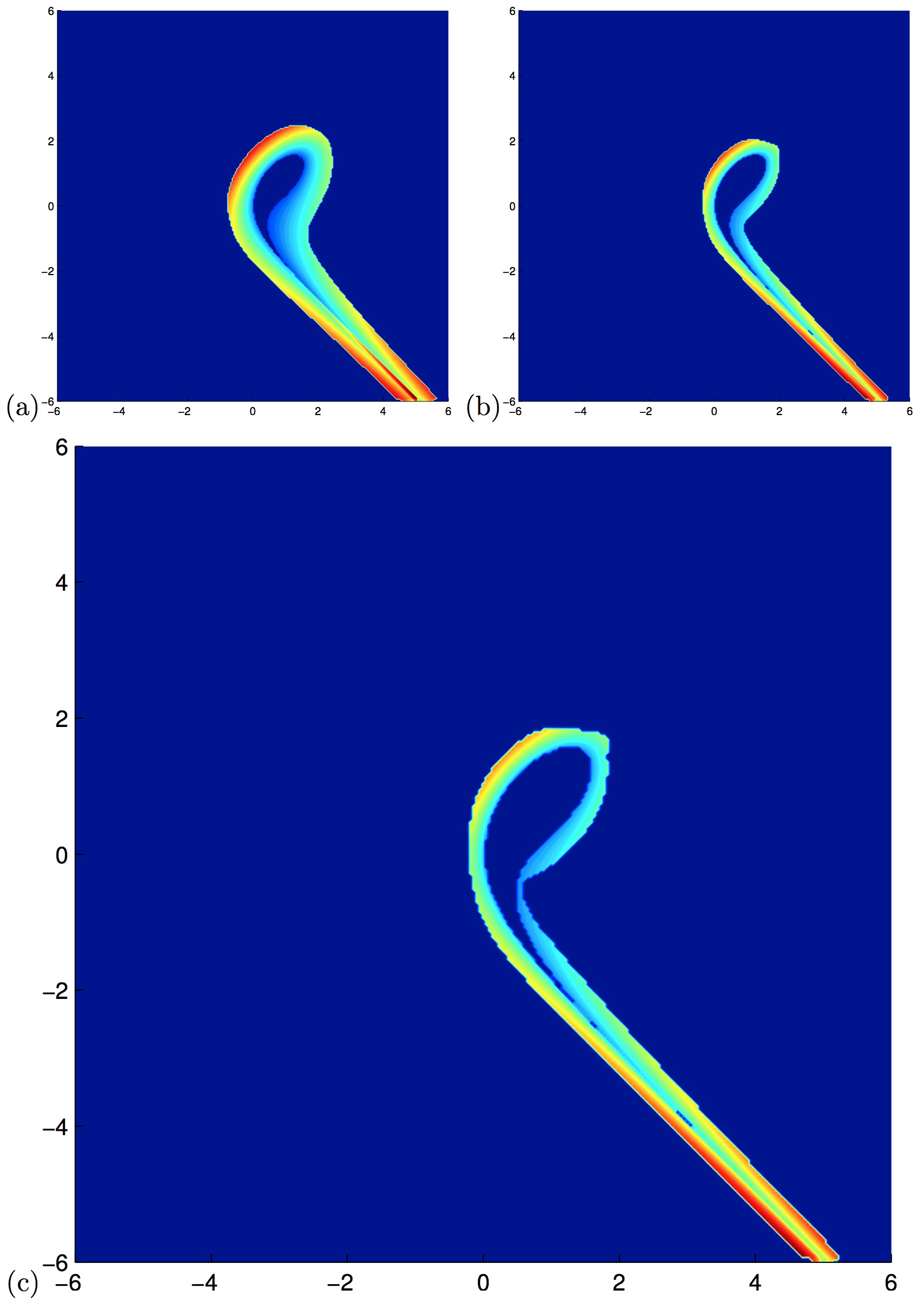}
}
\caption{(Section \ref{SubSec:Example_Analytic}) \reminderre{The computed ISLE field using the separation factor given by (a) $e^{3.5}$, (b) $e^4$ and (c) $e^{4.2}$. The outline of the ISLE ridge becomes clearer with an increasing separation factor.}}
\label{Fig:ISLE2a}
\end{figure}

\begin{figure}[!htb]
\centering{
\includegraphics[width=12cm]{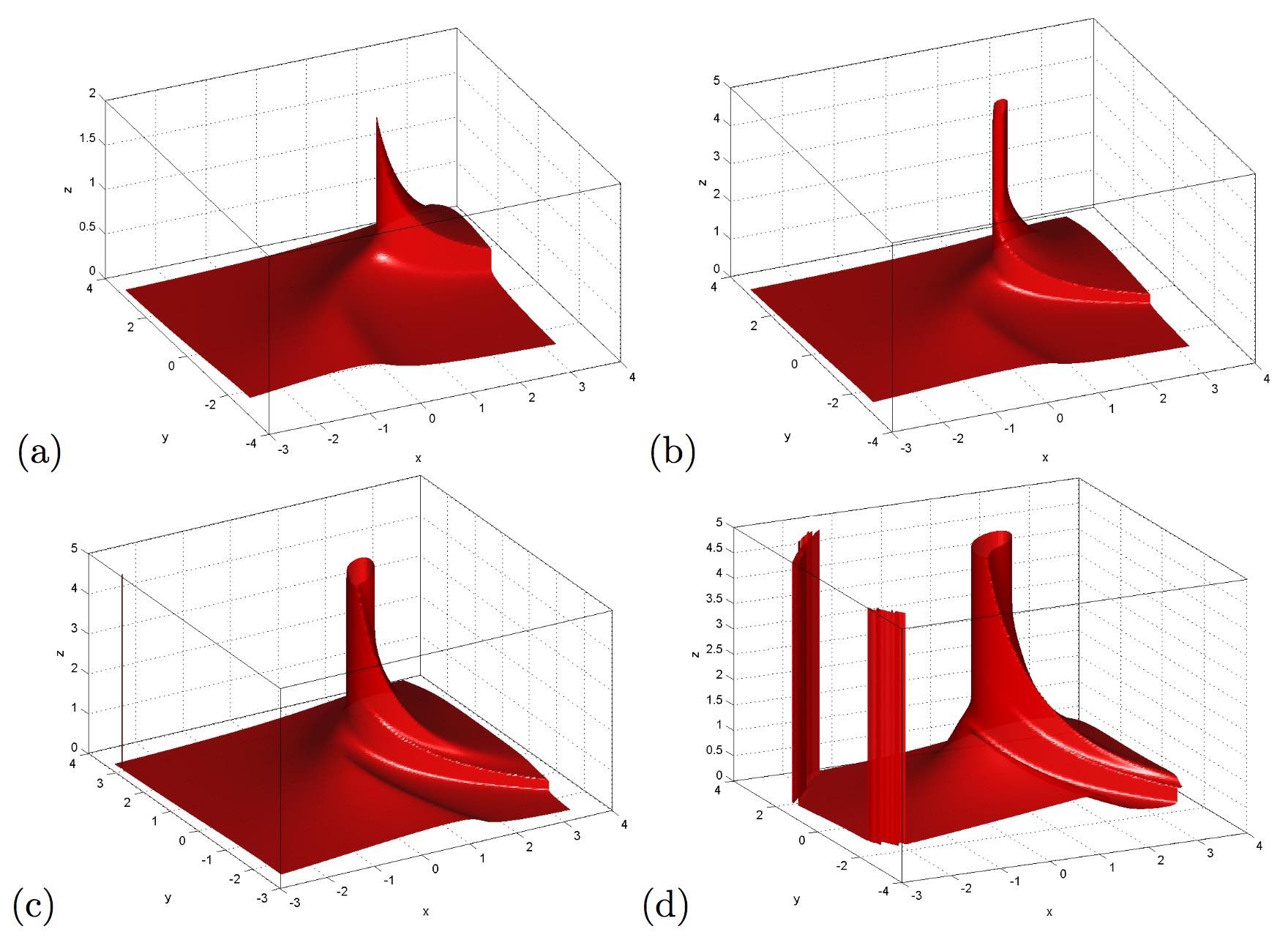}
}
\caption{(Section \ref{SubSec:Example_Analytic}) The isosurfaces of the ISLE field $\gamma_r(\mathbf{x},0)$ with (a) $r$=1, (b) $r$=2, (c) $r$=3 and (d) $r$=4. }
\label{Fig:ISLE2}
\end{figure}

We consider a simple analytical field \cite{tanchahal10} given by $u=x-y^2$ and $v=-y+x^2$ in this example. The computational domain is $[-6,6]^2$ and we discretize the domain with $\Delta x=\Delta y= 12/256$. We impose the fixed inflow boundary condition and non-reflective outflow boundary condition as in \cite{leu11}. Figure \ref{Fig:FTLE2} shows the FTLE field $\sigma_0^5(\mathbf{x})$ computed using our original Eulerian approach and the proposed Eulerian approach. It is clear that the new approach can successfully handle the inflow and outflow of particles on the boundaries, just like the previous approach does. Next, we are going to apply the results in Section \ref{Sec:Relation} to analyze the ISLE ridges. Figure \ref{Fig:FTLE2} suggests that there is only one FTLE ridge. And the minimum value of $\sigma_0^5$ on the ridge is close to $0.8$. By the definition of FTLE, we have $$
\sqrt{\lambda_0^5(\x)} = \exp\left[{\sigma_0^5(\x)\times 5}\right] \, .
$$
Therefore we can conclude that $e^{0.8 \times 5} = e^{4}$ is a reasonable estimate of the minimum value of $\sqrt{\lambda_0^5(\x)}$ on the ridge. We have tested with several separation factors that are near $e^4$, and the results can be plotted in Figure \ref{Fig:ISLE2a}. To better visualize the ISLE field, we plot in Figure \ref{Fig:ISLE2} the time $\tau_r(\mathbf{x})$ required for different separation factors $r$. The shorter the time, the bigger the corresponding ISLE.

\subsection{The forced-damped Duffing van der Pol equation}
\label{SubSec: DuffingVanDerPol}

\begin{figure} [!htb]
\centering{
\includegraphics[width=12cm]{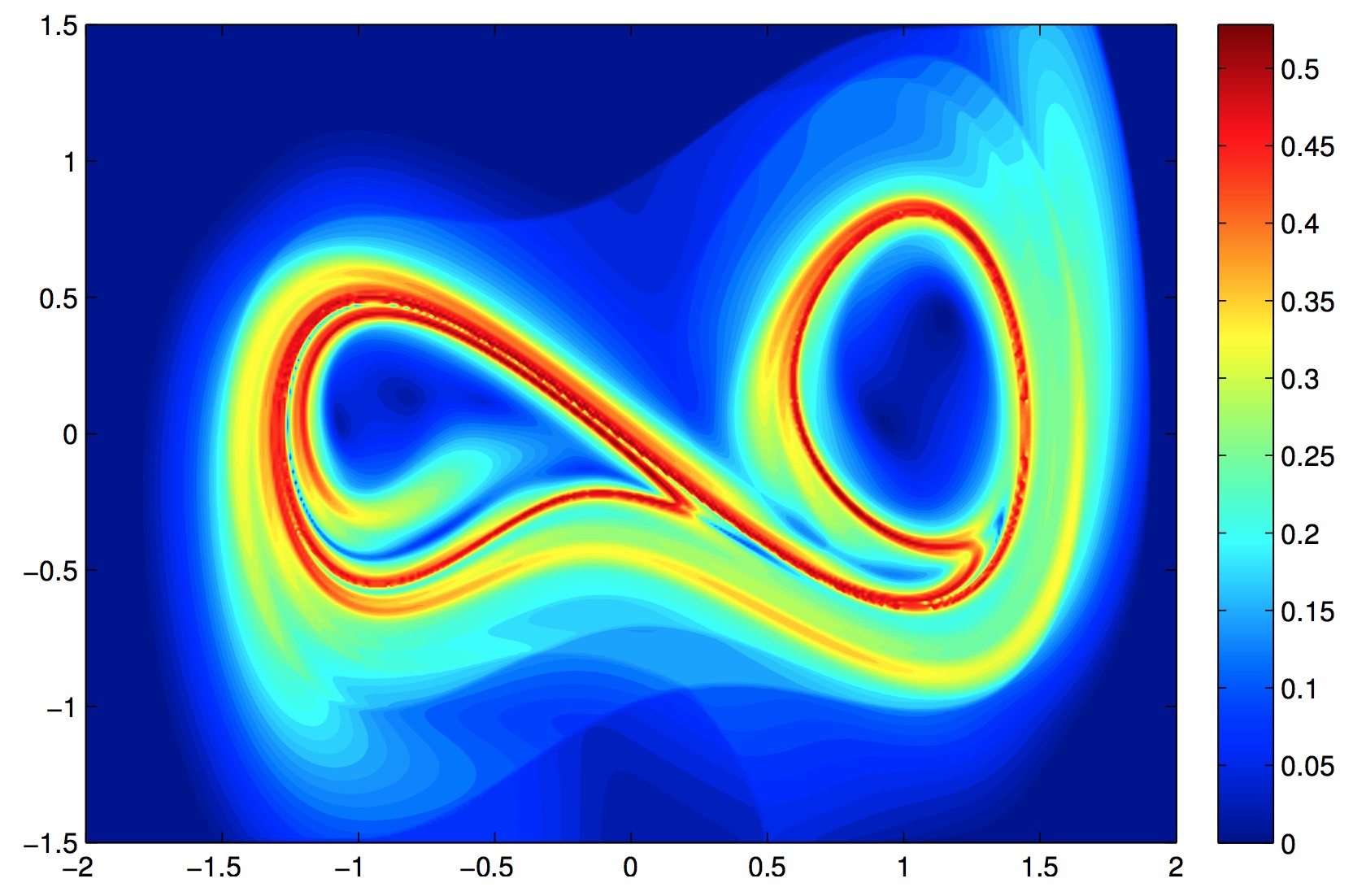}
}
\caption{(Section \ref{SubSec: DuffingVanDerPol}) The FTLE field $\sigma_0^{10}(\x)$ computed using our proposed Eulerian approach in Section \ref{SubSec:newEulerian}. For better comparison with the ISLE in Figure \ref{Fig:ISLE3}, we replace negative FTLE values by 0.}
\label{Fig:FTLE3}
\end{figure}

\begin{figure}[!htb]
\centering{
\includegraphics[width=12cm]{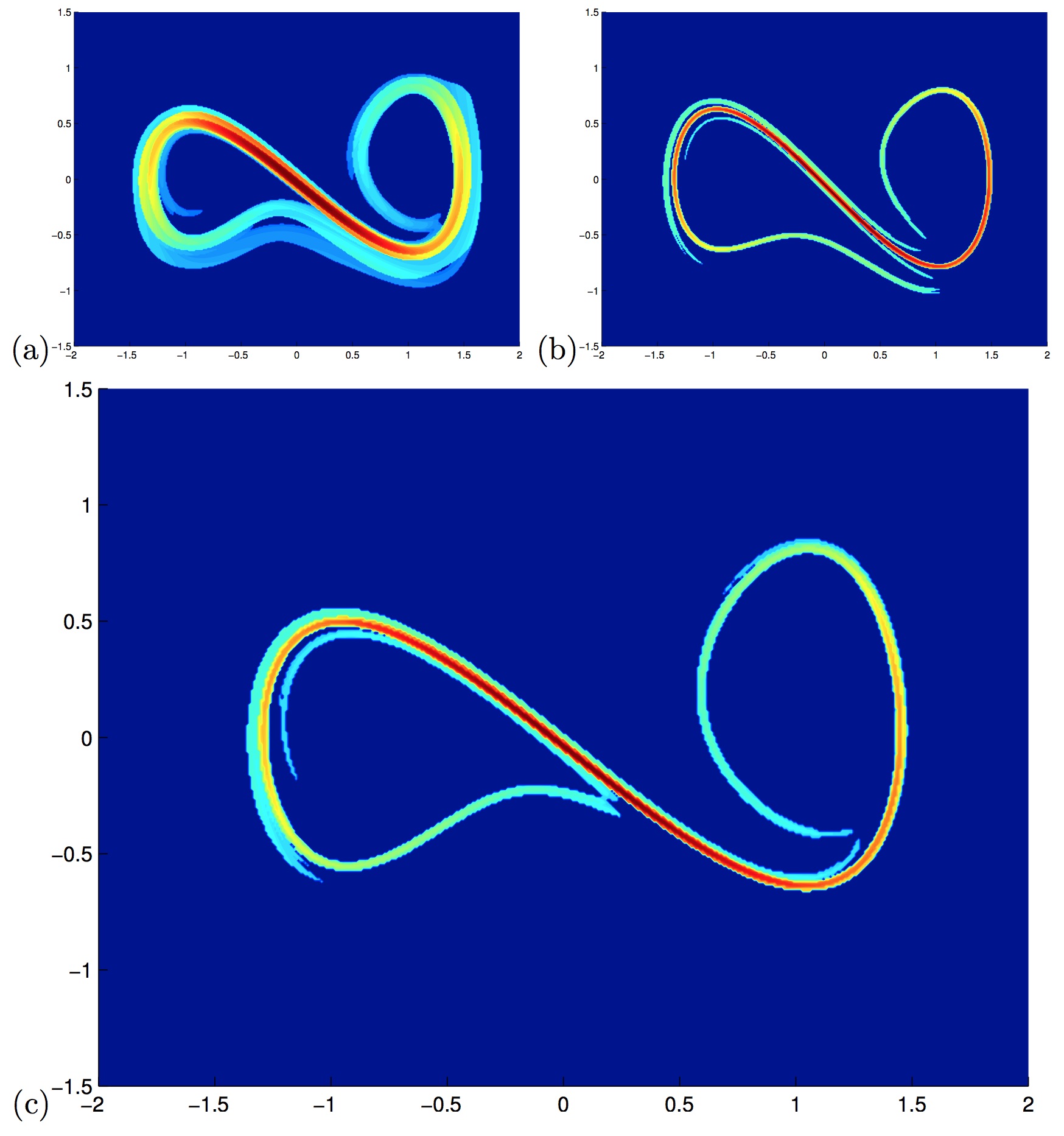}
}
\caption{(Section \ref{SubSec: DuffingVanDerPol}) The ISLE field $\gamma_r(\x)$ computed using our proposed Eulerian approach in Section \ref{SubSec:newEulerian} with (a) $r=e^3$, (b) $r=e^4$ and (c) $r=e^{4.1}$. For $r=e^{4.1}$, the ISLE ridges coincide with FTLE ridges in the Figure \ref{Fig:FTLE3}.}
\label{Fig:ISLE3}
\end{figure}

We consider the dynamical system governed by a Duffing and a van der Pol oscillator, as in \cite{halsap11}. The system is given by
$$
u=y \, , \, v=x-x^3+0.5y(1-x^2)+0.1 \sin t \, .
$$

The computational domain is $[-2,2]\times [-1.5,1.5]$, with mesh size $\Delta x = \Delta y = 0.01$. We simulate the flow from $t=0$ to $t=10$. The discretization size is rather coarse for such a complicated dynamics, but our new approach can still demonstrate the fine details of the transport barrier as shown in Figure \ref{Fig:FTLE3}. There are several major FTLE ridges that are easily spotted. The minimum value of $\sigma_0^t(\x)$ on those FTLE ridges are very close to $0.4$. So $\sqrt{\lambda_0^t(\x)} \geq e^{0.4 \times 10} = e^4$. \reminder{Theorem} \ref{prop:funProp1} implies that the separation factor $r=e^4$ can approximately single out the ISLE ridges. We pick several separation factors that are close to $e^{4}$, and we find that the ISLE ridges corresponding to $r=e^{4.1}$, as shown in Figure \ref{Fig:ISLE3}(c), have the greatest resemblance to the FTLE ridges in Figure \ref{Fig:FTLE3}.

\subsection{\reminder{The Ocean Surface Current Analyses Real-time (OSCAR) dataset}}
\label{SubSec:OSCAR}

\begin{figure} [!htb]
\centering{
\includegraphics[width=12cm]{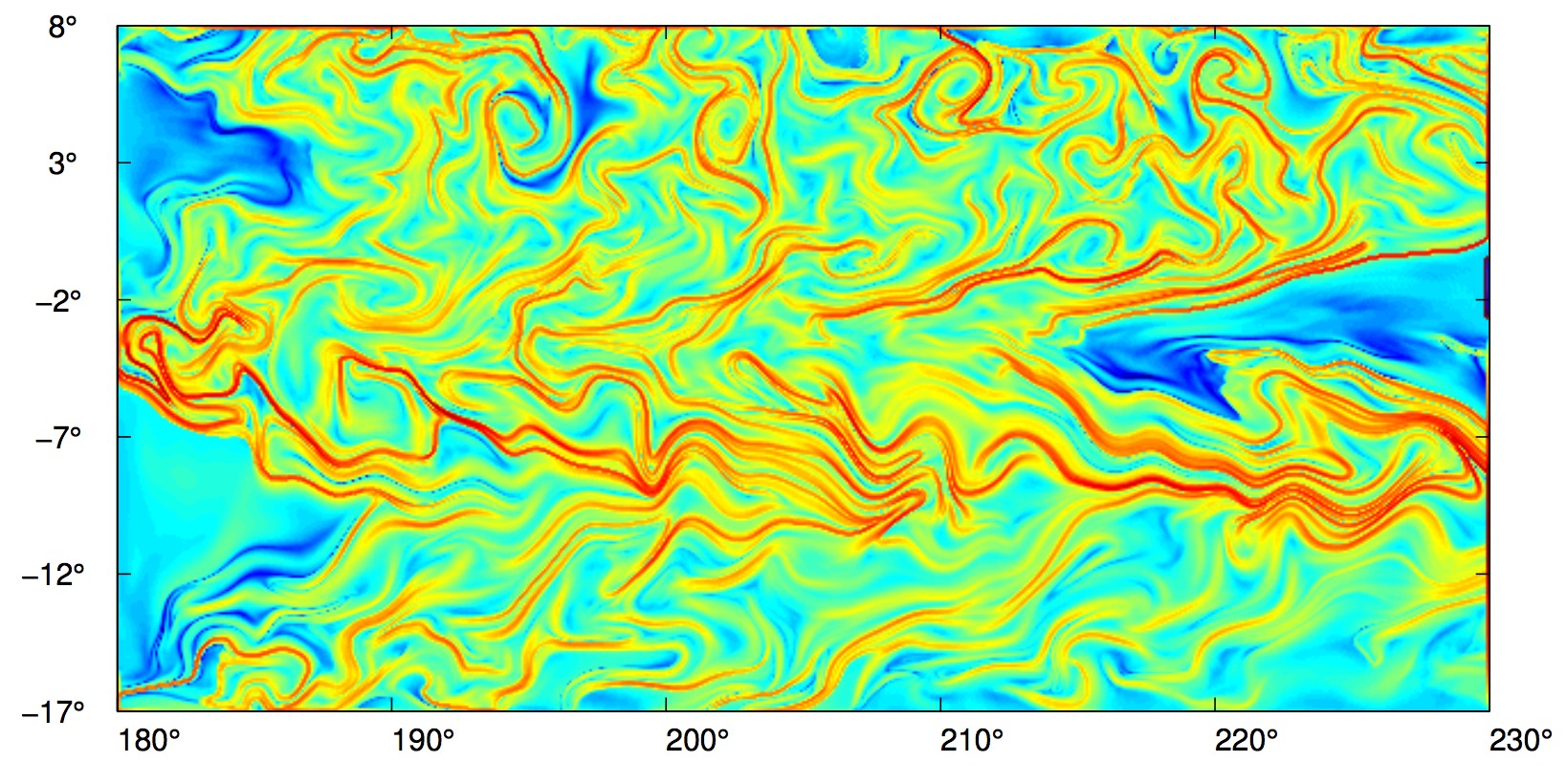}
}
\caption{\reminder{(Section \ref{SubSec:OSCAR}) The FTLE field $\sigma_0^{50}(\x)$ computed using our proposed Eulerian approach in Section \ref{SubSec:newEulerian}.}}
\label{Fig:OSCAR_FTLE}
\end{figure}

\begin{figure}[!htb]
\centering{
\includegraphics[width=12cm]{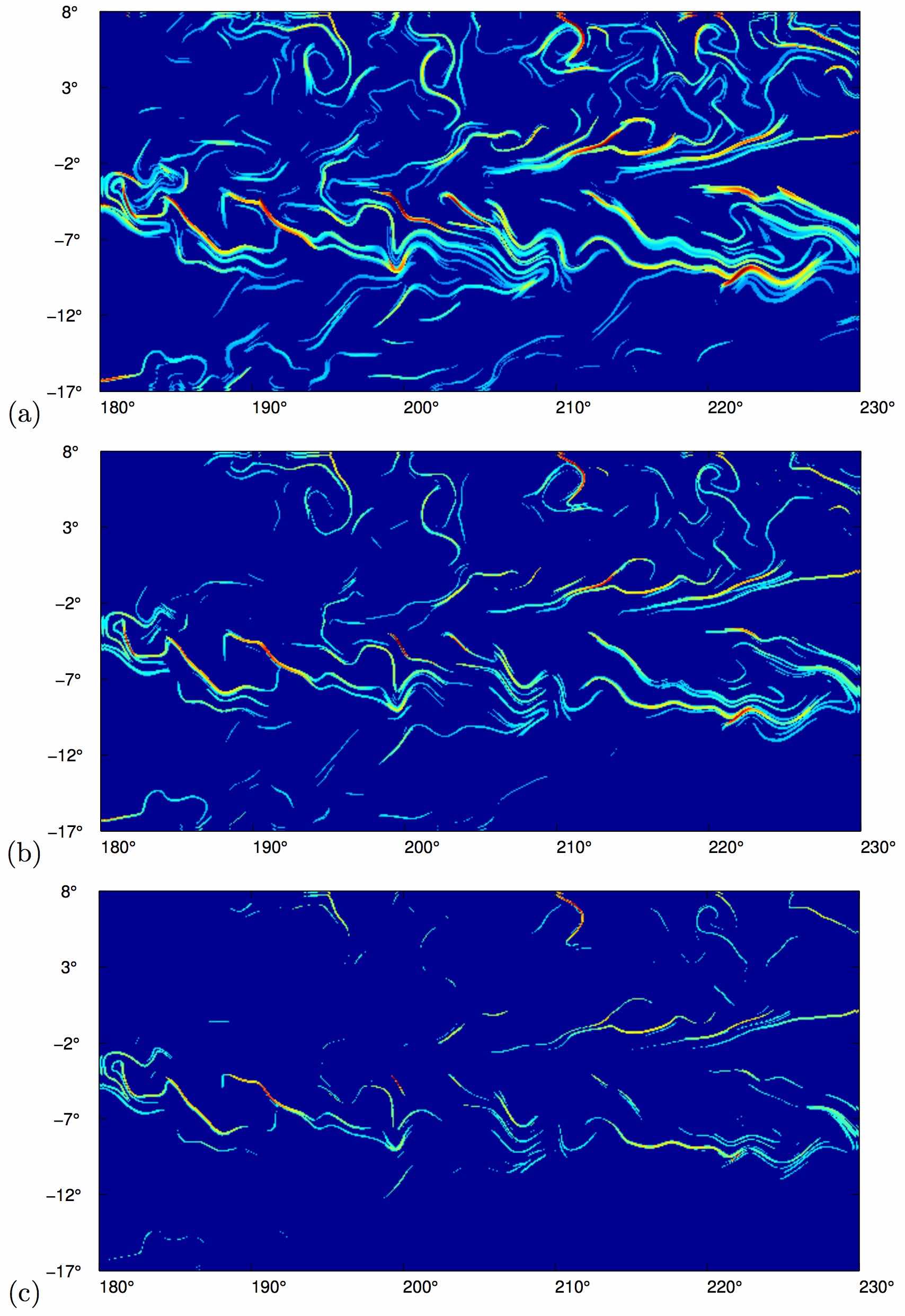}
\caption{\reminder{(Section \ref{SubSec:OSCAR}) The ISLE field $\gamma_r(\x)$ of OSCAR data computed using our proposed Eulerian approach in Section \ref{SubSec:newEulerian} with (a) $r=12\approx e^{2.5}$, (b) $r=20\approx e^{3}$, and (c) $r=30\approx e^{3.5}$.}}
\label{Fig:OSCAR_ISLE}
}
\end{figure}

\reminder{To further demonstrate the effectiveness and robustness of our proposed approach, we test our algorithm on a real life dataset obtained from the Ocean Surface Current Analyses Real-time (OSCAR)}\footnote{\url{http://www.esr.org/oscar_index.html}}. \reminder{The OSCAR data were obtained from JPL Physical Oceanography DAAC developed by ESR. It measures ocean surface flow in the region from $-80^{\circ}$ to $80^{\circ}$ in latitude and from $0^{\circ}$ to $360^{\circ}$ longitude. The time resolution is about 5 days and the spatial resolution is $1/3^{\circ}$ in each direction.}

\reminder{
In this numerical example, we consider the area near the Line Islands in the region from $-17^{\circ}$ to $8^{\circ}$ in latitude and from $180^{\circ}$ to $230^{\circ}$ longitude. To have a better visualization of both the FTLE and ISLE field, we interpolate the velocity data to obtain a finer resolution of $0.25$ days in the temporal direction and $1/12^{\circ}$ in each spatial direction. Then we look at the ocean surface current within the first 50 days in the year 2014 to obtain the forward FTLE field $\sigma_0^{50}(\x)$ and the ISLE field $\gamma_r(\x)$. }

 \reminder{
As demonstrated in Figure \ref{Fig:OSCAR_FTLE}, our new approach works well in extracting fine details of the transport barrier even for real data. There are a few major FTLE ridges spotted in the figure. According to Theorem \ref{prop:funProp1}, we found that the separation factor $r$ roughly equals to $e^{2.5}$ to $e^{3.5}$ can single out the ISLE ridges. We have tested several separation factors that are close to $e^3 \approx 20$, as shown in Figure \ref{Fig:OSCAR_ISLE}. We find that the ISLE ridges corresponding to $r=20$, in Figure \ref{Fig:OSCAR_ISLE}(b), have the greatest resemblance to the FTLE ridges in Figure \ref{Fig:OSCAR_FTLE}. As a comparison, we have also plotted the ISLE field $\gamma_r(\x)$ for $r=12$ and $r=30$ in Figure \ref{Fig:OSCAR_ISLE} (a) and (c), respectively. These results verify the theory developed in Section \ref{Sec:Relation}.}

\reminder{
\section{Conclusion}
To summarize, we emphasize the novelty and importance of both the two Eulerian algorithms we developed in this paper. Based on the first algorithm, we are now able to determine the \textit{forward} flow map \textit{on the fly} so that the PDE is solved \textit{forward} in time. Unlike the original Eulerian algorithm as developed in \cite{leu11,leu13}, there is no need to store the whole velocity field at all time steps at the beginning of the computations. This makes the forward flow map computations practical. More importantly, because of the first algorithm, we are now able to access all intermediate forward flow maps which is necessary in computing the ISLE field computations. This Eulerian approach is extremely efficient in computing the ISLE field. To the best of our knowledge, this is the first efficient numerical approach for the ISLE computations.
Theoretically, this paper provides a rigorous mathematical analysis to explain the similarity between the FTLE ridges and the ISLE ridges based on \cite{ppss14}, which has been verified by the proposed Eulerian approach.
}

\section*{Acknowledgment}
The work of You was supported by the Talents Introduction Project of Nanjing Audit University, and the Natural Science Foundation of Jiangsu Higher Education Institutions of China (No.16KJB110012). The work of Leung was supported in part by the Hong Kong RGC grants 16303114 and 16309316.

\bibliographystyle{plain}

\end{document}